\documentclass[reqno, 12pt]{amsart}
\usepackage{amsmath}
\usepackage{amsthm}
\usepackage{amsfonts}
\usepackage{amssymb}
\usepackage{mathrsfs}
\usepackage{graphicx}
\usepackage{slashed}

\newcommand{\mb}{\mathbf}
\newcommand{\mc}{\mathcal}

\renewcommand{\Re}{\mathrm{Re}\,}
\renewcommand{\Im}{\mathrm{Im}\,}
\newcommand{\rg}{\mathrm{rg}\,}
\newcommand{\N}{\mathbb{N}}
\newcommand{\R}{\mathbb{R}}
\newcommand{\C}{\mathbb{C}}
\newcommand{\Z}{\mathbb{Z}}
\newcommand{\B}{\mathbb{B}}
\renewcommand{\S}{\mathbb{S}}

\newcommand{\rad}{\mathrm{rad}}
\newcommand{\I}{\mathrm{i}\,}

\newtheorem{lemma}{Lemma}[section]
\newtheorem{proposition}[lemma]{Proposition}
\newtheorem{theorem}[lemma]{Theorem}
\newtheorem{corollary}[lemma]{Corollary}
\theoremstyle{remark}
\newtheorem{remark}[lemma]{Remark}

\theoremstyle{definition}
\newtheorem{definition}[lemma]{Definition}

\numberwithin{equation}{section}

\title{On blowup in supercritical wave equations}

\author{Roland Donninger}
\address{Rheinische Friedrich-Wilhelms-Universit\"at Bonn,
Mathematisches Institut, Endenicher Allee 60, D-53115 Bonn, Germany}
\email{donninge@math.uni-bonn.de}

\author{Birgit Sch\"orkhuber}
\address{Universit\"at Wien, Fakult\"at f\"ur Mathematik,
Oskar-Morgenstern-Platz 1, A-1090 Vienna, Austria}
\email{birgit.schoerkhuber@univie.ac.at}

\thanks{Roland Donninger is supported by the Alexander von 
Humboldt Foundation via a Sofja Kovalevskaja Award endowed by the German
Federal Ministry of Education and Research. 
Birgit Sch\"orkhuber is supported by the Austrian Science Fund
(FWF) via the Hertha Firnberg Program, Project Nr. T739-N25. Both authors would like to thank the Hausdorff Research
Institute for Mathematics in Bonn
for hospitality during the Trimester Program ``Harmonic Analysis and Partial Differential
Equations''.}

\begin{document}

\begin{abstract}
We study the blowup behavior
for the focusing energy-supercritical semilinear wave equation in $3$ space dimensions
without symmetry assumptions on the data. 
We prove the stability in $H^2\times H^1$ of the ODE blowup profile.
\end{abstract}

\maketitle

\tableofcontents

\section{Introduction}
\noindent The present paper is concerned with the Cauchy problem for the semilinear wave
equation
\begin{equation}
\label{eq:main}
 (-\partial_t^2 + \Delta_x)u(t,x)=-u(t,x)|u(t,x)|^{p-1},\qquad p>3 
 \end{equation}
for $x\in \R^3$.
Eq.~\eqref{eq:main} has the (indefinite) conserved energy
\[ \int_{\R^3}\left [ \tfrac12|\partial_t u(t,x)|^2 + \tfrac12|\nabla_x u(t,x)|^2
-\tfrac{1}{p+1}|u(t,x)|^{p+1} \right ]dx \]
and is invariant under the scaling transformation 
\[ u(t,x)\mapsto \lambda^{-\frac{2}{p-1}}
u(t/\lambda,x/\lambda),\qquad \lambda>0. \]
The corresponding scaling-invariant Sobolev space for the pair of functions
$(u(t,\cdot),
\partial_t u(t,\cdot))$ is $\dot H^{s_p}\times
\dot H^{s_p-1}(\R^3)$, where $s_p=\frac32-\frac{2}{p-1}$. 
Comparison with the free energy space $\dot H^1\times L^2(\R^3)$ shows that Eq.~\eqref{eq:main}
is energy-subcritical, critical, or supercritical if $p<5$, $p=5$, or $p>5$, respectively.
In fact, it is the latter case we are mainly interested in, although our results hold
for all $p>3$ (and with trivial modifications also for $p>1$).
Other symmetries of the equation
that are relevant in our context are time-translations 
and
Lorentz boosts. 
More precisely, 
for $a=(a^1,a^2,a^3)\in \R^3$ and $T\in \R$,
we set $\Lambda_T(a)=\Lambda_T^3(a^3)\Lambda_T^2(a^2)\Lambda_T^1(a^1)$, 
where $\Lambda^j: \R\times \R^3\to \R\times \R^3$, $j\in \{1,2,3\}$, is 
defined by 
\begin{align*}
\Lambda_T^j(a^j): \left \{ \begin{array}{rcl}
t &\mapsto &(t-T)\cosh a^j+x^j \sinh a^j+T \\
x^k &\mapsto &x^k+\delta^{jk}[(t-T)\sinh a^j+x^j \cosh a^j-x^j]
\end{array} \right . .
\end{align*}
Then, if $u$ is a solution to Eq.~\eqref{eq:main}, so is 
$u \circ \Lambda_T(a)$ for any $a\in \R^3$ and $T\in \R$
(the parameter $a$ is called \emph{rapidity}).
Note that $(T,0)$ is the point where the Lorentz transform $\Lambda_T(a)$ is anchored,
i.e., $(T,0)$ is a fixed point of the transformation $\Lambda_T(a)$ and the lightcones
emanating from $(T,0)$ are invariant under $\Lambda_T(a)$.

\subsection{Basic well-posedness theory}
By definition, $u$ is a \emph{(strong) solution} of Eq.~\eqref{eq:main} with data
$(f,g)\in H^s\times H^{s-1}(\R^3)$, $s\in \R$, 
if $u$ satisfies
\begin{align}
\label{eq:Duhamel}
u(t,\cdot)=&\cos(t|\nabla|)f
+\frac{\sin(t|\nabla|)}{|\nabla|}g \nonumber \\
&+\int_0^t \frac{\sin((t-s)|\nabla|)}{|\nabla|}
\left [u(s,\cdot)|u(s,\cdot)|^{p-1} \right ] ds
\end{align}
for $t$ in some interval containing $0$.
Based on this solution concept it follows
by Strichartz theory that Eq.~\eqref{eq:main} is locally well-posed in the 
critical Sobolev
space $\dot H^{s_p}\times \dot H^{s_p-1}(\R^3)$, see \cite{LinSog95, Sog08}. 
Furthermore,
local well-posedness in $H^2 \times H^1(\R^3)$ is classical, cf.~\cite{Tao06}.
In this respect it is worth noting that
Eq.~\eqref{eq:main} exhibits finite-time blowup from smooth, compactly
supported initial data, see below.
As a consequence, if $p>5$, the equation is ill-posed in the 
energy space $\dot H^1\times L^2$.
This is easily seen by rescaling a finite-time blowup solution \cite{Sog08, LinSog95}.
Indeed, due to the supercritical character, 
the lifespan of the rescaled solution decreases while
at the same time its energy decreases as well.
Consequently, there can be no small data local well-posedness in the energy space
and it is natural and necessary to study supercritical problems
in spaces of higher regularity.
In our case, $H^2\times H^1$ will suffice.

\subsection{Solutions in lightcones}
As a matter of fact, we need a more refined notion of solution
which allows us to localize to lightcones
\[ \Gamma(t_0,x_0):=\{(t,x)\in [0,t_0)\times \R^3: |x-x_0|\leq t_0-t\}, \]
where $(t_0,x_0)\in (0,\infty)\times \R^3$. 
If $u\in C^\infty(\Gamma(t_0,x_0))$ it is clear what it means that $u$ solves Eq.~\eqref{eq:main}.
However, we need to extend the notion of solution to functions $u$ on $\Gamma(t_0,x_0)$
with so little regularity that they cannot satisfy Eq.~\eqref{eq:main} in the 
sense of classical derivatives.
One way to do this is to use the Duhamel formula \eqref{eq:Duhamel} and a suitable
cut-off technique, see \cite{KilStoVis14}.
For our purposes it is more natural to resort to semigroup theory.
To motivate the following, suppose $u\in C^\infty(\Gamma(t_0,x_0))$ is a solution 
of Eq.~\eqref{eq:main} with $t_0>0$ and $x_0\in\R^3$.
A very natural coordinate system on the cone $\Gamma(t_0,x_0)$ is provided by
the \emph{self-similar variables}
\[ \tau=-\log(t_0-t)+\log t_0,\qquad \xi=\frac{x-x_0}{t_0-t}, \]
cf.~\cite{AntMer01, MerZaa03, MerZaa05,
MerZaa14a, MerZaa14b}.
Consequently, we introduce 
\begin{align} 
\label{eq:upsi}
\psi_1^{t_0,x_0}(\tau,\xi)&:=t_0^{\frac{2}{p-1}}e^{-\frac{2}{p-1}\tau}u(t_0-t_0e^{-\tau},
x_0+t_0 e^{-\tau}\xi) \nonumber \\
\psi_2^{t_0,x_0}(\tau,\xi)&:=t_0^{\frac{2}{p-1}+1}e^{-(\frac{2}{p-1}+1)\tau}
\partial_0 u(t_0-t_0e^{-\tau},
x_0+t_0 e^{-\tau}\xi).
\end{align}
This change of variables is discussed below in more detail.
It follows that $\psi_1^{t_0,x_0}, \psi_2^{t_0,x_0}\in C^\infty([0,\infty)\times 
\overline{\B^3})$
and Eq.~\eqref{eq:main} implies 
\begin{align}
\label{eq:syspsiintro}
\partial_0 \psi_1^{t_0,x_0}&=-\xi^j \partial_j \psi_1^{t_0,x_0}
-\tfrac{2}{p-1}\psi_1^{t_0,x_0}+\psi_2^{t_0,x_0} \nonumber \\
\partial_0 \psi_2^{t_0,x_0}&=\partial^j \partial_j \psi_1^{t_0,x_0}
-\xi^j \partial_j \psi_2^{t_0,x_0}-\tfrac{p+1}{p-1}\psi_2^{t_0,x_0}
+\psi_1^{t_0,x_0} |\psi_1^{t_0,x_0}|^{p-1}.
\end{align}
With $\mb{\tilde L}$ denoting the linear operator on the right-hand side of 
Eq.~\eqref{eq:syspsiintro} and $\mb N$ the nonlinearity, we can write this more succinctly as
\[ \partial_\tau \Psi^{t_0,x_0}(\tau)=\tilde{\mb L}\Psi^{t_0,x_0}(\tau)
+\mb N(\Psi^{t_0,x_0}(\tau)) \]
for $\Psi^{t_0,x_0}(\tau)=(\psi_1^{t_0,x_0}(\tau,\cdot),\psi_2^{t_0,x_0}(\tau,\cdot))$.
We will prove the following.

\begin{proposition}
\label{prop:main}
Let $\mc H:=H^2\times H^1(\B^3)$ and consider 
\[ \tilde{\mb L}: \mc D(\tilde{\mb L})\subset
\mc H \to \mc H \] 
with domain $\mc D(\tilde{\mb L}):=C^3\times C^2(\overline{\B^3})$.
Then $\tilde{\mb L}$ is densely defined, closable, and its closure generates a 
strongly-continuous one-parameter semigroup $\{\mb S(\tau): \tau\geq 0\}$
of bounded operators on $\mc H$.
\end{proposition}

As a consequence of Proposition \ref{prop:main} and Duhamel's principle, $\Psi^{t_0,x_0}$ satisfies
\begin{equation}
\label{eq:Duhamelintro}
\Psi^{t_0,x_0}(\tau)=\mb S(\tau)\Psi^{t_0,x_0}(0)+\int_0^\tau \mb S(\tau-\sigma)
\mb N(\Psi^{t_0,x_0}(\sigma))d\sigma
\end{equation}
for all $\tau\geq 0$.
These observations lead to the following natural
concept of solutions of Eq.~\eqref{eq:main} in lightcones.

\begin{definition}
We say that a function $u: \Gamma(t_0,x_0)\to \R$ is a \emph{solution} 
(more precisely, an \emph{$H^2$-solution})
of Eq.~\eqref{eq:main} if the corresponding $\Psi^{t_0,x_0}=(\psi_1^{t_0,x_0},
\psi_2^{t_0,x_0})$, given by Eq.~\eqref{eq:upsi}, 
belongs
to $C([0,\infty),H^2\times H^1(\B^3))$ and satisfies
Eq.~\eqref{eq:Duhamelintro}
for all $\tau\geq 0$.
\end{definition}

\subsection{Blowup surface}
Based on the Sobolev embedding $H^2(\B^3) \hookrightarrow L^\infty(\B^3)$ 
it is not hard to see that the nonlinearity $\mb N$ 
is locally Lipschitz on $H^2\times H^1(\B^3)$ and thus, 
it follows from Gronwall's inequality that for given data $(f,g) \in H^2\times H^1(\B^3_{t_0}(x_0))$,
there exists at most one solution $u: \Gamma(t_0,x_0)\to \R$ of Eq.~\eqref{eq:main}
with\footnote{We use the convenient abbreviation $u[t]:=(u(t,\cdot),\partial_t u(t,\cdot))$.} 
$u[0]=(f,g)$.
This leads to the concept of the \emph{blowup surface},
cf.~\cite{CafFri86, AntMer01, KilStoVis14, MerZaa14a, MerZaa14b}.

\begin{definition}
For given $x_0\in \R^3$ and $(f,g)\in H^2\times H^1(\R^3)$, we say that 
$t_0>0$ belongs to $A_{f,g}(x_0)\subset \R$
if there exists a solution $u: \Gamma(t_0,x_0)\to \R$ of Eq.~\eqref{eq:main} with
$u[0]=(f,g)|_{\B^3_{t_0}(x_0)}$.
We set 
\[ T_{f,g}(x_0):=\sup A_{f,g}(x_0) \cup \{0\}. \]
If $T_{f,g}(x_0)<\infty$, $T_{f,g}(x_0)$ is called the \emph{blowup time at $x_0$}.
The set
\[ \{(T_{f,g}(x),x) \in [0,\infty)\times \R^3: x\in \R^3\} \]
is called the \emph{blowup surface}.
\end{definition}

Consequently, to any data $(f,g)\in H^2\times H^1(\R^3)$ there exists a unique, 
maximally future-extended solution $u$ defined on a union of lightcones.
The domain of definition of $u$ is of the form $\{(t,x)\in [0,\infty)\times \R:
t<T_{f,g}(x)\}$.

\subsection{The main result}
With these preparations we are now ready to formulate our main result.
First, recall that
the existence of finite-time blowup is most easily seen by ignoring the
Laplacian in Eq.~\eqref{eq:main}. The remaining ODE in $t$ can be solved explicitly
which
 leads to the solution 
\[ u_1(t,x)=c_p (1-t)^{-\frac{2}{p-1}},\qquad 
c_p=\left [\frac{2(p+1)}{(p-1)^2}\right ]^\frac{1}{p-1}. \]
Obviously, $u_1$ becomes singular at $t=1$. 
One might object that this solution does not have data in $H^2\times H^1(\R^3)$. 
However, this
defect is easily fixed by using suitable cut-offs and exploiting finite speed of propagation.
By using symmetries of the equation one can in fact produce a much larger family of blowup
solutions.
For instance, the time translation symmetry immediately yields the one-parameter family
\[ u_T(t,x):=c_p (T-t)^{-\frac{2}{p-1}} \]
and by applying the Lorentz transform $\Lambda_T(a)$,
we can even generate the 4-parameter family 
\[ u_{T,a}(t,x):=u_T(\Lambda_T(a)(t,x)) \]
of explicit
blowup solutions. 
In this paper we prove the stability of the above family of explicit blowup solutions
in the following sense.

\begin{theorem}
\label{thm:main}Fix $p>3$.
There exist constants $M,\delta>0$ such that the following holds.
Suppose $(f,g)\in H^2\times H^1(\R^3)$ satisfy
\[ \|(f,g)-u_{1,0}[0]\|_{H^2\times H^1(\B^3_{1+\delta}(x_0))}\leq \tfrac{\delta}{M} \]
for some $x_0\in \R^3$.
Then  
$T:=T_{f,g}(x_0)\in [1-\delta,1+\delta]$ and there exists an
$a\in \B^3_{Mp\delta}$ such that 
the solution $u: \Gamma(T_{f,g}(x_0),x_0)\to \R$ of Eq.~\eqref{eq:main} with data 
$u[0]=(f,g)$
satisfies
\begin{align*} (T-t)^{-s_p+2}\|u[t]-u_{T,a}[t]
\|_{\dot H^2\times \dot H^1(\B^3_{T-t}(x_0))}&\lesssim 
(T-t)^{\frac{1}{p-1}} \\
(T-t)^{-s_p+1}\|u[t]-u_{T,a}[t]
\|_{\dot H^1\times L^2(\B^3_{T-t}(x_0))}&\lesssim 
(T-t)^{\frac{1}{p-1}} \\
(T-t)^{-s_p}\|u(t,\cdot)-u_{T,a}(t,\cdot)
\|_{L^2(\B^3_{T-t}(x_0))}&\lesssim 
(T-t)^{\frac{1}{p-1}} 
\end{align*}
for all $t\in [0,T)$, where $s_p=\frac32-\frac{2}{p-1}$ is the critical Sobolev exponent.
\end{theorem}

Slightly oversimplifying matters, Theorem \ref{thm:main} states that small perturbations
of the blowup solution $u_{1,0}(t,x)=c_p (1-t)^{-\frac{2}{p-1}}$ 
lead to solutions that converge back to
$u_{1,0}$ modulo symmetries of the equation.
Some remarks are in order.
\begin{itemize} 

\item The normalization factors in the estimates are natural given the fact that
\begin{align*}
\|u_{T,a}(t,\cdot)\|_{\dot H^k(\B_{T-t}^3)}&\simeq (T-t)^{s_p-k} \\
\|\partial_t u_{T,a}(t,\cdot)\|_{\dot H^k(\B_{T-t}^3)}&\simeq (T-t)^{s_p-1-k}
\end{align*}
for $k\in \N_0$ if $a\not=0$. 

\item The decay rates stated in Theorem \ref{thm:main} are not sharp. 
Inspection of the proof shows that one has in fact 
$(T-t)^{\frac{2}{p-1}-\epsilon}$ for any $\epsilon>0$ (the implicit constants depend on 
$\epsilon$, though).
It might even be possible to improve the decay rate to 
$(T-t)^{\frac{2}{p-1}}$, but this would
require a more detailed approach.

\item The topology in which we consider the problem is optimal in the class of Sobolev
spaces $H^k\times H^{k-1}$ with integer exponent $k$. Simple scaling arguments
show that it should be possible to lower the degree of regularity to $H^{s_p+\epsilon}
\times H^{s_p+\epsilon-1}$ for any $\epsilon>0$.
It is an intriguing question if a result like Theorem \ref{thm:main} can be proved
at the critical regularity level $H^{s_p}\times H^{s_p-1}$.
This is an open problem.

\item The assumption $p>3$ is not really needed and one can prove essentially the same result
for all $p>1$. However, in the subconformal range $p\in (1,3]$ a stronger statement
is true since one has stability in the weaker
topology $H^1\times L^2$. This was shown by Merle and Zaag \cite{MerZaa14a, MerZaa14b}.
\end{itemize}

\subsection{Related work}
There is a lot of activity in the study of blowup for wave equations.
Many recent works focus on energy-critical equations where type II blowup solutions
are studied that emerge from a dynamical rescaling of a soliton, 
e.g.~\cite{KenMer08, KriSchTat09, DonKri13, DonHuaKriSch14, KriSch14,
DuyKenMer12c, DuyKenMer13, DuyKenMer14b, DuyMer08, DuyKenMer11, DuyKenMer12b,
DuyKenMer12a, HilRap12}, see also
\cite{KriWon14, KriNakSch14, KriNakSch13a, KriNakSch13b}.

In the supercritical case, blowup is typically
self-similar.
The fact that the free energy cannot be used to control the nonlinearity causes serious
problems and thus,
a detailed investigation of the 
supercritical regime has only recently begun, see 
e.g.~\cite{DuyKenMer14a, KilVis11b, KilVis11a, KenMer11b, KenMer11a,
Bul12, Bul14, DodLaw14} for conditional global existence and scattering results. 
The recent work of Krieger and Schlag \cite{KriSch14a} 
provides a novel construction of large solutions
to the defocusing equation in the supercritical case.
The precise form of blowup was also investigated in a numerical study
by Bizo\'n, Chmaj, and Tabor \cite{BizChmTab04}.
Furthermore, we would like to mention a very exciting new development triggered by
a remarkable paper by Merle, Rapha\"el and Rodnianski on the nonlinear Schr\"odinger
equation \cite{MerRapRod14} 
which shows the existence of blowup
via a rescaling soliton in the supercritical regime for sufficiently high dimensions.
The corresponding installment for the wave
equation is due to Collot \cite{Col14}.  

In their fundamental work \cite{MerZaa03, MerZaa05}, 
Merle and Zaag proved the universality of the self-similar
blowup rate for Eq.~\eqref{eq:main} in the subconformal range $p\leq 3$, see also
\cite{KilStoVis14, HamZaa13} for blowup bounds in the full subcritical region $p<5$.
In the one-dimensional case (which is subcritical for all $p>1$) there
is a series of papers which provides a fairly complete picture 
of the blowup behavior \cite{MerZaa07, MerZaa08, MerZaa12a, MerZaa12b}.
Merle and Zaag were also able to extend some of these results to higher 
dimensions and in particular
they proved a result which is similar to ours but restricted to the
subconformal range, that is, $p\leq 3$ \cite{MerZaa14a, MerZaa14b}.
Furthermore, their stability result holds in the energy topology $H^1\times L^2$.
We would like to stress that under the same restriction $p\leq 3$, 
our method can easily be adapted to obtain stability
in $H^1\times L^2$ as well.
In fact, we have already proved this result for radial data \cite{DonSch12}.
We also remark in passing that for pure scaling reasons, 
stability in the energy topology can hold only if $p\leq 5$.
In this paper our main focus is on the supercritical range $p>5$ and this necessitates the
use of a stronger topology.

\subsection{Outline of the proof}
Our earlier works, where we have proved similar stability results 
\cite{DonSchAic11, Don11, DonSch12, DonSch14, Don14},
were restricted to 
\emph{spherically symmetric data}. 
In the present paper we do not impose any symmetry assumptions and this
requires a number of fundamentally new ideas compared to \cite{DonSch14}.
We outline the main steps of the proof.

\begin{itemize}
\item By translation symmetry, it suffices to consider the case $x_0=0$.
As in our earlier works, we use similarity coordinates 
\[ \tau=-\log(T-t)+\log T,\qquad \xi=\frac{x}{T-t}, \]
to study the evolution. Here, $T>0$ is a free parameter, to begin with. 
We write Eq.~\eqref{eq:main} as an abstract evolution problem of the form
\begin{equation}
\label{eq:introPsi}
 \partial_\tau \Psi(\tau)=\mb L\Psi(\tau)+\mb N(\Psi(\tau)) 
 \end{equation}
for a function $\Psi$ on $[0,\infty)$ with values in the Hilbert space 
$\mc H=H^2\times H^1(\B^3)$.
The (unbounded) linear operator $\mb L$ on $\mc H$ represents the free wave equation and $\mb N$
is the nonlinearity.
We reiterate that $\mb L$ is highly nonself-adjoint. From a heuristic point of view this is
clear since $\mb L$ is the generator of the free wave evolution in a backward lightcone.
Consequently, it must encode radiative properties which necessitates the existence
of nonreal spectrum.

\item We show that $\mb L$ generates a strongly-continuous semigroup $\mb S(\tau)$
on $\mc H$ with the sharp decay estimate $\|\mb S(\tau)\|_{H^2\times H^1}\lesssim e^{-\frac{2}{p-1}\tau}$.
This result follows by an application of the Lumer-Phillips Theorem.
In order to verify the hypothesis, we have to find an inner product $(\cdot|\cdot)$
on $\mc H$
which satisfies $\Re(\mb L\mb u|\mb u)\leq -\frac{2}{p-1}(\mb u|\mb u)$.
This is delicate since there is no straightforward way how to construct $(\cdot|\cdot)$.
We use a higher energy of the wave equation augmented with some carefully chosen boundary
terms at the lightcone.
Furthermore, one needs to show that the range of $\lambda-\mb L$ is dense in $\mc H$
for some $\lambda>-\frac{2}{p-1}$.
This boils down to solving a degenerate elliptic problem on the unit ball $\B^3$.
The degeneracy at the boundary $\partial \B^3=\S^2$ comes from the fact that 
$[0,\infty)\times \S^2$ is a characteristic surface (the boundary of a lightcone).
Consequently, standard elliptic theory is not applicable.
We proceed by a decomposition in spherical harmonics and reduce the problem to a system
of ordinary differential equations.

 \item Next, we insert the modulation ansatz $\Psi(\tau)=\Psi_{a(\tau)}+\Phi(\tau)$
 into Eq.~\eqref{eq:introPsi}, where $\Psi_a$ corresponds to the blowup solution
 $u_{T,a}$ in the new coordinates.
 Thus, we allow the rapidity to depend on $\tau$.
 There is no dependence on $T$ since the coordinate transformation $(t,x)\mapsto (\tau,\xi)$
 is $T$-dependent and chosen in such a way that Eq.~\eqref{eq:introPsi} and $\Psi_a$ are 
 independent of $T$.
 The modulation ansatz leads to an equation of the form
 \begin{align*}
 \partial_\tau  \Phi(\tau)-\mb L \Phi(\tau)-\mb L'_{a(\tau)}\Phi(\tau)
 =\mb N_{a(\tau)}(\Phi(\tau)) 
 -\partial_\tau \Psi_{a(\tau)} 
 \end{align*}
 where $\mb N_{a(\tau)}(\mb u)=\mb N(\Psi_{a(\tau)}+\mb u)-\mb N(\Psi_{a(\tau)})
 -\mb L'_{a(\tau)}\mb u$ and $\mb L'_{a(\tau)}$ is the linearization of $\mb N$
 at $\Psi_{a(\tau)}$.
We cannot deal with a $\tau$-dependent ``potential term'' on the left-hand side and therefore,
assuming $\lim_{\tau\to\infty}a(\tau)=a_\infty$, we rewrite this equation as
\begin{align}
\label{eq:introPhi}
\partial_\tau  &\Phi(\tau)-\mb L \Phi(\tau)-\mb L'_{a_\infty}\Phi(\tau) \nonumber \\
 &=[\mb L'_{a(\tau)}-\mb L'_{a_\infty}]\Phi(\tau)+\mb N_{a(\tau)}(\Phi(\tau)) 
 -\partial_\tau \Psi_{a(\tau)} .
\end{align}
Very conveniently, the operator $\mb L'_{a_\infty}$ turns out to be compact.
We remark that the particular version of modulation theory we use is inspired by the work
of Krieger and Schlag on the critical wave equation \cite{KriSch07}.

\item By the Bounded Perturbation Theorem it immediately follows that $\mb L+\mb L'_{a_\infty}$
generates a semigroup $\mb S_{a_\infty}(\tau)$, but in order to obtain a useful growth bound,
a detailed spectral analysis of the generator 
$\mb L_{a_\infty}:=\mb L+\mb L'_{a_\infty}$ is necessary.
Fortunately, after decomposition in spherical harmonics,
the spectral equation for $\mb L_0$ can be solved explicitly in terms of
hypergeometric functions. 
By a spectral-theoretic perturbation argument we obtain sufficient
information on the spectrum of $\mb L_{a_\infty}$, provided $a_\infty$ is small.
In particular, we find that $\mb S_{a_\infty}(\tau)$ has a $4$-dimensional unstable subspace
which is spanned by the generators of time translations and Lorentz boosts.
Transversal to this unstable subspace, the semigroup $\mb S_{a_\infty}(\tau)$ decays
exponentially.
We construct a suitable spectral projection onto the unstable subspace 
that commutes with the semigroup.

\item By using Duhamel's principle, we rewrite Eq.~\eqref{eq:introPhi} as an integral
equation of the form
\begin{align}
\label{eq:introPhiweak}
 \Phi(\tau)=&\mb S_{a_\infty}(\tau)\Phi(0) \nonumber \\
 &+\int_0^\tau \mb S_{a_\infty}(\tau-\sigma)
 \big [ \mb L'_{a(\sigma)}-\mb L'_{a_\infty} \big ]\Phi(\sigma)d\sigma \nonumber \\
 &+\int_0^\tau \mb S_{a_\infty}(\tau-\sigma)
 \big [\mb N_{a(\sigma)}(\Phi(\sigma))-\partial_\sigma \Psi_{a(\sigma)}\big ]d\sigma.
 \end{align}
 The nonlinear terms are controlled by standard Sobolev embeddings and
we construct a decaying solution to Eq.~\eqref{eq:introPhiweak} 
by a nested fixed point argument.
More precisely,
we first derive a suitable equation for the modulation parameters $a(\tau)$ and 
via a fixed point
argument we show that $a(\tau)$ can be chosen in such a way that the instability
due to the Lorentz symmetry gets suppressed.
For the time translation symmetry we use a different approach and employ a version of
the Lyapunov-Perron method where one modifies the equation by a suitable correction
term in order to stabilize the evolution.
The modified equation is then solved by a fixed point argument.
In a second step one shows that the correction term vanishes provided one chooses
the blowup time $T$ correctly.
This is done by a Brower-type argument.
\end{itemize}

\subsection{Notation}
The arguments for functions defined on Minkowski space are numbered by $0,1,2,3$
and we write $\partial_0$, $\partial_1$, $\partial_2$, $\partial_3$ for the respective derivatives.
We use the notation $\partial_y$ for the derivative with respect to the variable $y$.
We employ Einstein's summation convention throughout with latin indices running from $1$
to $3$, unless otherwise stated. 
The symbol $\B^d_R(x_0)$ denotes the open ball of radius $R$ 
in $\R^d$, centered at $x_0\in \R^d$.
For brevity we set $\B^d_R:=\B^d_R(0)$, $\B^d:=\B^d_1$, 
and $\S^{d-1}:=\partial \B^d$.

The letter $C$ (possibly with indices to indicate dependencies) denotes a generic
positive constant which may have a different value at each occurrence.  
The symbol $a\lesssim b$ means $a\leq Cb$ and we abbreviate $a\lesssim b\lesssim a$
by $a\simeq b$. We write $f(x)\sim g(x)$ for $x\to a$ if $\lim_{x\to a}\frac{f(x)}{g(x)}=1$.

For a closed linear operator $\mb L$ on a Banach space we denote its domain by $\mc D(\mb L)$, its spectrum
by $\sigma(\mb L)$, and its point spectrum by $\sigma_p(\mb L)$.
We write $\mb R_{\mb L}(z):=(z-\mb L)^{-1}$ for $z\in \rho(\mb L)=\C
\backslash \sigma(\mb L)$.
The space of bounded operators on a Banach space $\mc X$ is denoted by $\mc B(\mc X)$.

\section{Preliminary transformations}

\noindent We start with the wave equation
\begin{equation}
\label{eq:wavep}
 (-\partial_t^2+\Delta_x)u(t,x)=-u(t,x)|u(t,x)|^{p-1},\quad p>3. 
 \end{equation}
The aforementioned $4$-parameter family of blowup solutions $u_{T,a}$ is explicitly given by
\begin{equation}
\label{eq:blowupsol} 
u_{T,a}(t,x)=c_p\big [A_0(a)(T-t)-A_j(a)x^j]^{-\frac{2}{p-1}}
\end{equation}
where $a=(a^1,a^2,a^3) \in \R^3$ and
\begin{align*}
A_0(a)&=\cosh a^1 \cosh a^2 \cosh a^3 \\
A_1(a)&=\sinh a^1 \cosh a^2 \cosh a^3 \\
A_2(a)&=\sinh a^2 \cosh a^3 \\\
A_3(a)&=\sinh a^3.
\end{align*}
Our intention is to study the time evolution of small perturbations of $u_{1,0}$.
Thus, we consider the Cauchy problem
\begin{equation}
\label{eq:Cauchy}
\left \{ \begin{array}{l}
(-\partial_t^2+\Delta_x)u(t,x)=-u(t,x)|u(t,x)|^{p-1},\quad t>0 \\
u[0]=u_{1,0}[0]+(\tilde f,\tilde g)
\end{array} \right.
\end{equation}
where the functions $\tilde f, \tilde g$ are small
in a suitable sense.

Now we introduce the similarity variables 
\[ \tau=-\log(T-t)+\log T,\qquad \xi=\frac{x}{T-t}. \]
The derivatives transform according to
\[ \partial_t=\frac{e^\tau}{T}(\partial_\tau+\xi^j\partial_{\xi^j}),
\quad \partial_{x^j}=\frac{e^\tau}{T} \partial_{\xi^j} \]
and we obtain
\[ \partial_t^2=\frac{e^{2\tau}}{T^2}(\partial_\tau^2+\partial_\tau+2\xi^j\partial_{\xi^j}\partial_\tau
+\xi^j \xi^k \partial_{\xi^j}\partial_{\xi^k}+2\xi^j\partial_{\xi^j}) \]
as well as $\partial_{x_j}\partial_{x^j}=\frac{e^{2\tau}}{T^2}\partial_{\xi_j}\partial_{\xi^j}$.
Consequently, Eq.~\eqref{eq:wavep} is equivalent to
\begin{align*}
\big [\partial_\tau^2&+\partial_\tau+2\xi^j\partial_{\xi^j}\partial_\tau
-(\delta^{jk}-\xi^j \xi^k)\partial_{\xi^j}\partial_{\xi^k}+2\xi^j\partial_{\xi^j}\big ]
U(\tau,\xi) \\
&=T^2e^{-2\tau}U(\tau,\xi)|U(\tau,\xi)|^{p-1}
\end{align*}
where $U(\tau,\xi)=u(T-Te^{-\tau},Te^{-\tau}\xi)$.
In order to remove the $\tau$-dependent factor on the right-hand side we rescale
and set 
\[ U(\tau,\xi)=T^{-\frac{2}{p-1}}e^{\frac{2}{p-1}\tau}\psi(\tau,\xi) \] which yields
\begin{align}
\label{eq:psi}
\Big [\partial_\tau^2&+\tfrac{p+3}{p-1}\partial_\tau+2\xi^j\partial_{\xi^j}\partial_\tau
-(\delta^{jk}-\xi^j \xi^k)\partial_{\xi^j}\partial_{\xi^k} \nonumber \\
&+2\tfrac{p+1}{p-1}\xi^j \partial_{\xi^j}
+2\tfrac{p+1}{(p-1)^2}\Big ]\psi(\tau,\xi) 
=\psi(\tau,\xi)|\psi(\tau,\xi)|^{p-1}.
\end{align}
By applying the above transformations to the blowup solution $u_{T,a}$ defined in 
Eq.~\eqref{eq:blowupsol},
we obtain the $3$-parameter family $\psi_a$ of \emph{static} solutions to Eq.~\eqref{eq:psi}
given by
\begin{equation}
\label{eq:psi_a}
 \psi_a(\xi)=c_p \left [A_0(a)-A_j(a)\xi^j\right ]^{-\frac{2}{p-1}}. 
 \end{equation}
Note that the dependence on $T$ has 
disappeared since we have used a coordinate transformation adapted to the blowup time $T$.
As a consequence, the dependence on $T$ is hidden as a 
symmetry of the equation and will reappear later as an instability
of the evolution one has to deal with. 
We also remark that for $a \in \R^3$ small, the solution $\psi_a$ is smooth since $A_0(a)=1+O(|a|)$
and $A_j(a)=O(|a|)$
if $|a|\lesssim 1$.

In order to rewrite Eq.~\eqref{eq:psi} as a first-order system in time, we introduce the variables
\[ \psi_1:=\psi,\qquad \psi_2(\tau,\xi)
:=\partial_\tau \psi(\tau,\xi)+\xi^j\partial_j \psi(\tau,\xi)+\tfrac{2}{p-1}\psi(\tau,\xi) \]
where the definition of $\psi_2$ is motivated by the fact that $\psi_2$ 
corresponds to the $t$-derivative of $u$, i.e.,
\[ \partial_0 u(T-Te^{-\tau},Te^{-\tau}\xi)=T^{-\frac{p+1}{p-1}}e^{\frac{p+1}{p-1}\tau}\big [\partial_\tau
+\xi^j\partial_{\xi^j}+\tfrac{2}{p-1}\big ]\psi(\tau,\xi). \]
Then Eq.~\eqref{eq:psi} transforms into the system
\begin{align}
\label{eq:syspsi}
\partial_0 \psi_1&=-\xi^j \partial_j \psi_1-\tfrac{2}{p-1}\psi_1+\psi_2 \nonumber \\
\partial_0 \psi_2&=\partial^j \partial_j \psi_1-\xi^j \partial_j \psi_2-\tfrac{p+1}{p-1}\psi_2
+\psi_1 |\psi_1|^{p-1}
\end{align}
and from Eq.~\eqref{eq:psi_a} we derive the family of static solutions
\begin{align}
\label{eq:syspsi_a} 
\psi_{a,1}(\xi)&=c_p \left [A_0(a)-A_j(a)\xi^j\right ]^{-\frac{2}{p-1}} \nonumber \\
\psi_{a,2}(\xi)&=\frac{2c_p}{p-1}A_0(a) 
\left [A_0(a)-A_j(a)\xi^j\right ]^{-\frac{p+1}{p-1}}.
\end{align}
By noting that $A_0(a)^2-A_j(a)A^j(a)=1$ one may also check directly that 
$(\psi_{a,1},\psi_{a,2})$
is indeed a solution of Eq.~\eqref{eq:syspsi}.
For the initial data we obtain
\begin{align}
\label{eq:data}
\psi_1(0,\xi)&=T^{\frac{2}{p-1}}[\psi_{0,1}(T\xi)+\tilde f(T\xi)] \nonumber \\
\psi_2(0,\xi)&=T^{\frac{p+1}{p-1}}[\psi_{0,2}(T\xi)+\tilde g(T\xi)].
\end{align}
We emphasize that the only trace of the parameter $T$ is in the initial data.

\section{A semigroup formulation for the linear evolution}

\subsection{Function spaces}
We will make use of semigroup theory to treat the evolution problem Eq.~\eqref{eq:syspsi}.
A key ingredient is the construction of a suitable 
inner product on $H^2 \times H^1(\B^3)$ that yields the sharp
decay for the free evolution.
We start by considering the two sesquilinear forms 
\begin{align*}
(\mb u | \mb v)_1&:=\int_{\B^3}\partial^k \partial^j u_1 \overline{\partial_k \partial_j v_1}
+\int_{\B^3}\partial^j u_2 \overline{\partial_j v_2}
+\int_{\S^2}\partial^j u_1\overline{\partial_j v_1}d\sigma \\
(\mb u | \mb v)_2&:=\int_{\B^3}\Delta u_1 \overline{\Delta v_1}
+\int_{\B^3}\partial^j u_2 \overline{\partial_j v_2}+\int_{\S^2}u_2\overline{v_2} d\sigma
\end{align*}
on the space
$\tilde{\mc H}:=C^2(\overline{\B^3})
\times C^1(\overline{\B^3})$, where $\sigma$ denotes the surface measure on the 2-sphere.
Obviously, both sesquilinear forms are derived from a standard (higher) energy of the free
wave equation but neither of them defines an inner product on $\tilde{\mc H}$.
This will be fixed later. 
Next, we define an operator that represents the free part in Eq.~\eqref{eq:syspsi}
on the right-hand side.
We set 
\[ \tilde{\mb L}\mb u(\xi):=\left (\begin{array}{c}
-\xi^j \partial_j u_1(\xi)-\frac{2}{p-1}u_1(\xi)+u_2(\xi) \\
\partial^j \partial_j u_1(\xi)-\xi^j\partial_j u_2(\xi)-\frac{p+1}{p-1}u_2(\xi) 
\end{array} \right ). \]
At this point, $\tilde{\mb L}$ should be viewed as a formal differential operator;
we will specify a suitable domain later.
The key properties of the sesquilinear forms $(\cdot|\cdot)_1$ and $(\cdot|\cdot)_2$
are stated in the following lemma.

\begin{lemma}
\label{lem:sesq}
We have the estimates
\[ \Re(\tilde{\mb L}\mb u | \mb u)_j\leq -(\tfrac{2}{p-1}+\tfrac12)(\mb u|\mb u)_j \]
for all $\mb u \in C^3(\overline{\B^3})\times C^2(\overline{\B^3})$ and $j \in \{1,2\}$.
\end{lemma}

\begin{proof}
We write $[\tilde{\mb L} \mb u]_j$, $j\in \{1,2\}$, for the $j$-th component of $\tilde{\mb L}\mb u$.
Since 
\[ 2\Re[\xi^j \partial_j f(\xi)\overline{f(\xi)}]=\partial_{\xi^j}[\xi^j|f(\xi)|^2]-3|f(\xi)|^2 \]
and
\[ \partial_{\xi_k} \partial_{\xi_j} [\xi^i \partial_i f(\xi)]=2\partial^k \partial^j f(\xi)
+\xi^i \partial_i \partial^k \partial^j f(\xi), \] the
divergence theorem yields
\begin{align*}
\Re \int_{\B^3}\partial^k \partial^j [\tilde{\mb L}\mb u]_1\overline{\partial_k\partial_j u_1}
&=-\left (\tfrac{2}{p-1}+\tfrac12 \right )
\int_{\B^3}\partial^k\partial^j u_1 \overline{\partial_k \partial_j u_1} \\
&\quad -\tfrac12 \int_{\S^2}\partial^k \partial^j u_1(\omega)
\overline{\partial_k \partial_j u_1(\omega)} d\sigma(\omega) \\
&\quad +\Re \int_{\B^3}\partial^k \partial^j u_2 \overline{\partial_j \partial_k u_1}.
\end{align*}
Analogously, we find
\begin{align*}
\Re \int_{\B^3}\partial^j [\tilde{\mb L}\mb u]_2\overline{\partial_j u_2}&=
-\Re \int_{\B^3}\partial^k \partial^j u_1\overline{\partial_k\partial_j u_2} \\
&\quad +\Re \int_{\S^2}\omega^j \partial_j \partial^k u_1(\omega)
\overline{\partial_k u_2(\omega)}d\sigma(\omega) \\
&\quad -\left (\tfrac{2}{p-1}+\tfrac12 \right )\int_{\B^3}\partial^j u_2 \overline{\partial_j u_2} \\
&\quad -\tfrac12 \int_{\S^2}\partial^j u_2(\omega)\overline{\partial_j u_2(\omega)}d\sigma(\omega).
\end{align*}
For the boundary term we obtain
\begin{align*}
\Re \int_{\S^2}\partial^j [\tilde{\mb L}\mb u]_1(\omega)\overline{\partial_j u_1(\omega)}d\sigma(\omega)
&=-\left (\tfrac{2}{p-1}+1\right )\int_{\S^2}\partial^j u_1\overline{\partial_j u_1}d\sigma(\omega) \\
&\quad -\Re\int_{\S^2}\omega^k \partial_k \partial^j u_1(\omega) \overline{\partial_j u_1(\omega)}d\sigma(\omega) \\
&\quad +\Re\int_{\S^2}\partial^j u_2(\omega)\overline{\partial_ju_1(\omega)}d\sigma(\omega).
\end{align*}
Summing up, we infer
\[ \Re(\tilde{\mb L}\mb u|\mb u)_1\leq -(\tfrac{2}{p-1}+\tfrac12)(\mb u|\mb u)_1+\int_{\S^2}
A(\omega)d\sigma(\omega), \]
where 
\begin{align*}
A(\omega)&=-\tfrac12 \partial^k \partial^j u_1(\omega)\overline{\partial_k \partial_j u_1(\omega)}
-\tfrac12 \partial^j u_1(\omega)\overline{\partial_j u_1(\omega)} \\
&\quad -\tfrac12 \partial^j u_2(\omega)\overline{\partial_j u_2(\omega)}  \\
&\quad -\Re \left [\omega^k \partial_k \partial^j u_1(\omega)\overline{\partial_j u_1(\omega)} \right ]
+ \Re \left [\omega^k \partial_k \partial^j u_1(\omega)\overline{\partial_j u_2(\omega)} \right ]  \\
&\quad 
+\Re \left [\partial^j u_1(\omega)\overline{\partial_j u_2(\omega)}\right ].
\end{align*}
Now we use the inequality 
\[ \Re(a\overline b)+\Re(a\overline c)-\Re(b\overline c)\leq \tfrac12 |a|^2
+\tfrac12 |b|^2+\tfrac12 |c|^2, \qquad a,b,c \in \C, \]
which follows from $0\leq |a-b-c|^2$, to obtain
\begin{align*}
-\Re&\left [\omega^k \partial_k \partial_j u_1 \overline{\partial_j u_1} \right ]
+\Re \left [\omega^k \partial_k \partial_j u_1 \overline{\partial_j u_2} \right ]
+\Re \left [\partial_j u_1 \overline{\partial_j u_2} \right ] \\
&\leq \tfrac12 \left |\omega^k \partial_k \partial_j u_1 \right |^2
+\tfrac12 \left | \partial_j u_1 \right |^2+\tfrac12 \left | \partial_j u_2 \right |^2 \\
&\leq \tfrac12 \omega^{\ell} \omega_\ell \partial^k \partial_j u_1 
\overline{\partial_k \partial_j u_1}
+\tfrac12 \left | \partial_j u_1 \right |^2+\tfrac12 \left | \partial_j u_2 \right |^2 \\
&=\tfrac12 \partial^k \partial_j u_1 
\overline{\partial_k \partial_j u_1}
+\tfrac12 \partial_j u_1 \overline{\partial_j u_1}
+\tfrac12 \partial_j u_2 \overline{\partial_j u_2}
\end{align*}
for each $j\in \{1,2,3\}$.
Consequently, summation over $j$ yields 
$A(\omega)\leq 0$
and we arrive at the desired estimate for $(\tilde{\mb L}\mb u|\mb u)_1$.

In order to prove the claimed estimate for $(\tilde{\mb L}\mb u | \mb u)_2$,
we note that
\begin{align*}
\Re \int_{\B^3}\Delta [\tilde{\mb L}\mb u]_1 \overline{\Delta u_1}
&=-\Re \int_{\B^3}\xi^j\partial_j \Delta u_1(\xi) \overline{\Delta u_1(\xi)}d\xi \\
&\quad -(\tfrac{2}{p-1}+2)\int_{\B^3}|\Delta u_1|^2
+\Re \int_{\B^3}\Delta u_2 \overline{\Delta u_1} \\
&=-\tfrac12 \int_{\S^2}|\Delta u_1|^2 d\sigma
-(\tfrac{2}{p-1}+\tfrac12)\int_{\B^3}|\Delta u_1|^2 \\
&\quad +\Re \int_{\B^3}\Delta u_2 \overline{\Delta u_1}
\end{align*}
and
\begin{align*}
\Re  \int_{\B^3}\partial^j [\tilde{\mb L}\mb u]_2 \overline{\partial_j u_2}&=
\Re \int_{\B^3}\partial^j \Delta u_1\overline{\partial_j u_2} \\
&\quad -\Re \int_{\B^3}\xi^k \partial_k \partial^j u_2(\xi)\overline{\partial_j
u_2(\xi)}d\xi \\
&\quad - (\tfrac{2}{p-1}+2)\int_{\B^3}\partial^j u_2 \overline{\partial_j u_2} \\
&=-\Re \int_{\B^3}\Delta u_1 \overline{\Delta u_2} \\
&\quad +\Re\int_{\S^2}\Delta u_1(\omega)\overline{\omega^j \partial_j u_2(\omega)}
d\sigma(\omega) \\
&\quad -\tfrac12 \int_{\S^2}\partial^j u_2 \overline{\partial_j u_2}d\sigma
- (\tfrac{2}{p-1}+\tfrac12)\int_{\B^3}\partial^j u_2 \overline{\partial_j u_2}.
\end{align*}
The boundary term yields
\begin{align*}
\Re \int_{\S^2}[\tilde{\mb L}\mb u]_2 \overline{u_2}d\sigma&=
\Re \int_{\S^2}\Delta u_1 \overline{u_2}d\sigma -\Re\int_{\S^2}
\omega^j \partial_j u_2(\omega)\overline{u_2(\omega)}d\sigma(\omega) \\
&\quad -(\tfrac{2}{p-1}+1)\int_{\S^2}|u_2|^2 d\sigma.
\end{align*}
In summary, we obtain
\[
\Re(\tilde{\mb L}\mb u | \mb u)_2\leq -(\tfrac{2}{p-1}+\tfrac12)(\mb u|\mb u)_2
+\int_{\S^2}B(\omega)d\sigma(\omega) \]
where
\begin{align*}
B(\omega)&=-\tfrac12 |\Delta u_1(\omega)|^2 -\tfrac12 \partial^j u_2(\omega)
\overline{\partial_j u_2(\omega)}-\tfrac12 |u_2(\omega)|^2 \\
&\quad +\Re \left [ \Delta u_1(\omega)\overline{\omega^j\partial_j u_2(\omega)}
\right ]
+\Re \left [ \Delta u_1(\omega)\overline{u_2(\omega)} \right ] \\
&\quad -\Re \left [ \omega^j \partial_j 
u_2(\omega)\overline{u_2(\omega)}\right ] \\
&\leq 0.
\end{align*}
\end{proof}

In addition, we consider a third sesquilinear form which constitutes the missing piece for
a proper inner product. We set
\[ (\mb u|\mb v)_3:=\int_{\S^2}[\omega^j \partial_j u_1+u_1+u_2](\omega)d\sigma(\omega)
\int_{\S^2}\overline{[\omega^j\partial_j v_1+v_1+v_2](\omega)}d\sigma(\omega).
 \]
The following ``magic'' property (which, by the way, only works in 3 space
dimensions) is crucial.

\begin{lemma}
\label{lem:sesq3}
We have
\[ \Re(\tilde{\mb L}\mb u | \mb u)_3\leq -\tfrac{2}{p-1}(\mb u | \mb u)_3 \]
for all $\mb u \in C^3(\overline{\B^3})\times C^2(\overline{\B^3})$.
\end{lemma}

\begin{proof}
We have
\begin{align*}
\int_{\S^2}&\left [\omega^j \partial_j [\tilde{\mb L}\mb u]_1 + [\tilde{\mb L}\mb u]_1+
[\tilde{\mb L}\mb u]_2 \right ](\omega)d\sigma(\omega) \\
&=-\tfrac{2}{p-1}\int_{\S^2}\left [ \omega^j \partial_j u_1+u_1+u_2 \right ](\omega)d\sigma(\omega) \\
&\quad + \int_{\S^2} \left [ (\delta^{jk}-\omega^j \omega^k)\partial_j \partial_k u_1-2\omega^j
\partial_j u_1 \right ](\omega)d\sigma(\omega) \\
&=-\tfrac{2}{p-1}\int_{\S^2}\left [ \omega^j \partial_j u_1+u_1+u_2 \right ](\omega)d\sigma(\omega)
\end{align*}
since $-[(\delta^{jk}-\omega^j \omega^k)\partial_j \partial_k
-2\omega^j \partial_j]$ is the Laplace-Beltrami
operator on $\S^2$.
\end{proof}

Now we set
\[ (\mb u | \mb v):=\sum_{j=1}^3 (\mb u | \mb v)_j \]
and claim that $\|\cdot\|:=\sqrt{(\cdot | \cdot)}$ is equivalent to
$\|\cdot\|_{H^2 \times H^1(\B^3)}$.
In order to prove this, it is useful to recall the following result.

\begin{lemma}
\label{lem:trace}
We have
\[ \|f\|_{H^1(\B^3)}\simeq \|\nabla f\|_{L^2(\B^3)}+\|f\|_{L^2(\S^2)} \]
for all $f\in C^1(\overline{\B^3})$.
\end{lemma}

\begin{proof}
See e.g.~\cite{DonZen14}, Lemma 3.1.
\end{proof}

Furthermore, we need a Poincar\'e-type inequality.

\begin{lemma}
\label{lem:Poincare}
Let
\[ Sf:=\frac{1}{4\pi}\int_{\S^2}f(\omega)d\sigma(\omega). \]
Then we have
\[ \|f-Sf\|_{L^2(\B^3)}\lesssim \|\nabla f\|_{H^1(\B^3)} \]
for all $f\in C^2(\overline{\B^3})$.
\end{lemma}

\begin{remark}
We emphasize the differences of Lemma \ref{lem:Poincare} and the standard Poincar\'e inequality:
We take the mean over the sphere $\S^2$ instead of the ball $\B^3$ and on the right-hand side
we have the $H^1$-norm of the gradient instead of the $L^2$-norm. 
\end{remark}

\begin{proof}[Proof of Lemma \ref{lem:Poincare}]
The statement is obviously true if
$\|\nabla f\|_{H^1(\B^3)}=0$, because then we have $f(x)=c$ for some constant
$c\in \C$ which implies $f(x)-Sf=0$ for all $x\in \overline{\B^3}$.
Thus, we may exclude this trivial case.
As in the classical proof of the Poincar\'e inequality, we now argue by contradiction.
Suppose the statement were false. Then we could find a sequence 
$(f_n)\subset C^2(\overline{\B^3})$ with $\|\nabla f_n\|_{H^1(\B^3)}>0$
and such that 
\[ \|f_n-Sf_n\|_{L^2(\B^3)}\geq n \|\nabla f_n\|_{H^1(\B^3)} \]
for all $n\in \N$.
We set
\[ g_n:=\frac{f_n-Sf_n}{\|f_n-Sf_n\|_{L^2(\B^3)}}. \]
Then we have $\|g_n\|_{L^2(\B^3)}=1$, $Sg_n=0$, and $\|\nabla g_n\|_{H^1(\B^3)}\leq \frac{1}{n}$.
In particular, $\|g_n\|_{H^2(\B^3)}\lesssim 1$. By the compactness of the embedding
$H^2(\B^3)\hookrightarrow H^1(\B^3)$ (and after passing to a subsequence, if necessary), we
find a function $g \in H^1(\B^3)$ with $g_n\to g$ in $H^1(\B^3)$ and $\|g\|_{L^2(\B^3)}=1$.
It follows that
\begin{align*} \|g_n-g_m\|_{H^2(\B^3)}&\lesssim \|g_n-g_m\|_{L^2(\B^3)}+
\|\nabla g_n-\nabla g_m\|_{H^1(\B^3)} \\
&\lesssim \|g_n-g_m\|_{L^2(\B^3)}+\tfrac{1}{n}+\tfrac{1}{m}
\end{align*}
and thus, $g_n \to g$ even in $H^2(\B^3)$.
By Sobolev embedding we infer $g\in C(\overline{\B^3})$
and from this it follows that $Sg=0$.
Furthermore, we have $\|\nabla g\|_{L^2(\B^3)}=0$ and thus, $g(x)=c$ for almost every
$x\in \B^3$ and some $c\in \C$.
Since $g\in C(\overline{\B^3})$, we obtain $g(x)=c$ for all $x\in \overline{\B^3}$
and from $Sg=0$ we infer $c=0$.
This, however, contradicts $\|g\|_{L^2(\B^3)}=1$.
\end{proof}

\begin{lemma}
\label{lem:ip}
The sesquilinear form $(\cdot|\cdot)$ defines an inner product on $\tilde{\mc H}$.
The completion of $\tilde{\mc H}$ is a Hilbert space which is equivalent to
$H^2\times H^1(\B^3)$.
\end{lemma}

\begin{proof}
We have to show that $\|\mb u\|\simeq \|\mb u\|_{H^2\times H^1(\B^3)}$ for all 
$\mb u\in \tilde{\mc H}$.
The estimate $\|\mb u\|\lesssim \|\mb u\|_{H^2 \times H^1(\B^3)}$ is a simple consequence
of the trace theorem.
Thus, it suffices to prove $\|\mb u\|_{H^2\times H^1(\B^3)}\lesssim \|\mb u\|$.
From Lemma \ref{lem:trace} we obtain
$\|u_2\|_{H^1(\B^3)}\lesssim \|\mb u\|$ and
$\|\nabla u_1\|_{H^1(\B^3)}\lesssim \|\mb u\|$.
Furthermore, Lemma \ref{lem:Poincare} yields
\begin{align*} 
\|u_1\|_{L^2(\B^3)}&\leq \|u_1-Su_1\|_{L^2(\B^3)}+\|Su_1\|_{L^2(\B^3)} \\
&\lesssim \|\nabla u_1\|_{H^1(\B^3)}+\left |\int_{\S^2}u_1(\omega)d\sigma(\omega) \right | \\
&\lesssim \|\mb u\|+\left |
\int_{\S^2}[\omega^j\partial_j u_1+u_1+u_2](\omega)d\sigma(\omega) \right | \\
&\quad +\|\nabla u_1\|_{L^2(\S^2)}+\|u_2\|_{L^2(\S^2)} \\
&\lesssim \|\mb u\|
\end{align*}
which finishes the proof.
\end{proof}

\subsection{Generation of the semigroup}
In order to obtain a proper functional analytic setup, we augment the formal
operator $\tilde{\mb L}$ with the domain
$\mc D(\tilde{\mb L})=C^3\times C^2(\overline{\B^3})$.
With this domain, $\tilde{\mb L}: \mc D(\tilde{\mb L})\subset \mc H \to \mc H$, where
$\mc H:=H^2 \times H^1(\B^3)$, is a well-defined (unbounded) operator
on $\mc H$.
Our goal is to show that $\tilde{\mb L}$ generates a semigroup (which then 
governs the free wave evolution in a backward lightcone).

\begin{proposition}
\label{prop:free}
The operator $\tilde{\mb L}: \mc D(\tilde{\mb L})\subset \mc H\to\mc H$ is densely defined,
closable, and its closure $\mb L$ generates a strongly-continuous one-parameter semigroup
$\mb S: [0,\infty)\to \mc B(\mc H)$ which satisfies
\[ \|\mb S(\tau)\|\leq e^{-\frac{2}{p-1}\tau} \]
for all $\tau\geq 0$.
\end{proposition}

The proof consists of an application of the Lumer-Phillips Theorem.
In order to verify the hypothesis, we have to show that the range of
$\lambda-\tilde{\mb L}$ is dense in $\mc H$, for some $\lambda>-\frac{2}{p-1}$
(it turns out that $\lambda=\frac32-\frac{2}{p-1}$ is convenient).
This requires solving a degenerate elliptic problem
and we do this by an angular momentum decomposition.
For the convenience of the reader and to fix notation we have compiled
the necessary background material 
in Appendix \ref{sec:apx}.
Based on this, we prove the following technical lemma which is the
key ingredient for the aforementioned density result.

\begin{lemma}
\label{lem:degen}
Let $f\in H^1(\B^3)$ and $\varepsilon>0$. 
Then there exists a function $u\in C^3(\overline{\B^3})$ such that 
$g\in C^1(\overline{\B^3})$, defined by
\[ g(\xi):=-(\delta^{jk}-\xi^j \xi^k)\partial_j \partial_k u(\xi)
+5\xi^j \partial_j u(\xi)+\tfrac{15}{4}u(\xi), \]
satisfies $\|f-g\|_{H^1(\B^3)}<\varepsilon$.
\end{lemma}

\begin{proof}
By the density of $C^\infty(\overline{\B^3})$ in $H^1(\B^3)$ we find a function
$\tilde g \in C^\infty(\overline{\B^3})$ satisfying 
$\|\tilde g-f\|_{H^1(\B^3)}<\frac{\varepsilon}{2}$.
For $r \in (0,1]$ 
we set $g_{\ell,m}(r):=(\tilde g(r\,\cdot)|Y_{\ell,m})_{\S^2}$
where 
$Y_{\ell,m} \in C^\infty(\S^2)$ for $ \ell\in \N_0$ and $m \in \{-\ell,-\ell+1,\dots,\ell-1,\ell\}$
denote the standard (orthonormalized) spherical harmonics on $\S^2$.
Then we define $g_N \in C^\infty(\overline{\B^3}\backslash \{0\})$ by
\[ g_N(\xi):=\sum_{\ell=0}^N \sum_{m=-\ell}^\ell g_{\ell,m}(|\xi|)
Y_{\ell,m}(\tfrac{\xi}{|\xi|}). \]
Lemma \ref{lem:expfull} shows that $g_N$ extends to a function
in $C^\infty(\overline{\B^3})$ 
and that we may choose $N\in \N$ so large
that $\|g_N-\tilde g\|_{H^1(\B^3)}<\frac{\varepsilon}{2}$.
The lemma is proved if we can find a function 
$u\in C^3(\overline{\B^3})$ satisfying
\begin{equation}
\label{eq:ellipticred} -(\delta^{jk}-\xi^j \xi^k)\partial_j \partial_ku(\xi)
+5\xi^j\partial_j u(\xi)+\tfrac{15}{4}u(\xi)=g_N(\xi) 
\end{equation}
since $\|g_N-f\|_{H^1(\B^3)}\leq \|g_N-\tilde g\|_{H^1(\B^3)}
+\|\tilde g-f\|_{H^1(\B^3)}<\varepsilon$.
With polar coordinates $\xi=\rho\omega$, 
$(\rho,\omega)\in (0,1]\times \S^2$, we obtain
\begin{align*}
-&(\delta^{jk}-\xi^j \xi^k)\partial_j \partial_k u(\rho\omega)
+5\xi^j\partial_j u(\rho\omega) +\tfrac{15}{4}u(\rho\omega) \\
&=\left [-(1-\rho^2)\partial_\rho^2-\tfrac{2}{\rho}\partial_\rho
  +5\rho\partial_\rho-\tfrac{1}{\rho^2}\slashed\partial_{\omega_j}
  \slashed\partial_{\omega^j}+\tfrac{15}{4} \right ]u(\rho\omega)
 \end{align*}
 where $-\slashed\partial^j\slashed\partial_j$ is the Laplace-Beltrami operator
 on $\S^2$, see Appendix \ref{sec:apx}.
Consequently, the ansatz
\begin{equation}
\label{eq:ansatz}
 u(\rho\omega)=\sum_{\ell=0}^N \sum_{m=-\ell}^\ell u_{\ell,m}(\rho)
Y_{\ell,m}(\omega) 
\end{equation}
yields the decoupled system
\begin{equation}
\label{eq:ode} 
\left [-(1-\rho^2)\partial_\rho^2-\tfrac{2}{\rho}\partial_\rho
+5\rho\partial_\rho+\tfrac{\ell(\ell+1)}{\rho^2}+\tfrac{15}{4}\right ]
u_{\ell,m}(\rho)=g_{\ell,m}(\rho)
\end{equation}
of ODEs for the functions $u_{\ell,m}$.
We first consider the homogeneous version of this equation, i.e., we
set $g_{\ell,m}$ equal to zero.
Suppressing the subscripts, we define a new dependent variable $v$ by 
setting $u_{\ell,m}(\rho)=\rho^\ell v(\rho^2)$.
Then Eq.~\eqref{eq:ode} (with $g_{\ell,m}=0$) is equivalent to
the hypergeometric differential equation
\[ z(1-z)v''(z)+[c-(a+b+1)z]v'(z)-ab v(z)=0 \]
where $z=\rho^2$ and 
$a=\tfrac{3+2\ell}{4}$, $b=\tfrac{5+2\ell}{4}$, $c=\tfrac{3+2\ell}{2}$.
For the following facts on hypergeometric functions we refer to
\cite{DLMF}.
With ${}_2F_1$ denoting the standard hypergeometric function, 
we have the two solutions
\begin{align*} 
\phi_0(z)&={}_2F_1(a,b,c;z)={}_2F_1(\tfrac{3+2\ell}{4}, \tfrac{5+2\ell}{4},
\tfrac{3+2\ell}{2}; z) \\
\phi_1(z)&={}_2F_1(a,b,a+b+1-c; 1-z)
={}_2F_1(\tfrac{3+2\ell}{4}, \tfrac{5+2\ell}{4}, \tfrac32; 1-z) 
\end{align*}
which are analytic around $z=0$ and $z=1$, respectively.
Furthermore, a third solution is given by
\begin{align*} \tilde \phi_1(z)&=(1-z)^{c-a-b}{}_2F_1(c-a, c-b, c-a-b+1,1-z) \\
&=(1-z)^{-\frac12}{}_2F_1(\tfrac{3+2\ell}{4},\tfrac{1+2\ell}{4},\tfrac12; 1-z) 
\end{align*}
which is singular at $z=1$.
We set $\psi_j(\rho):=\rho^\ell \phi_j(\rho^2)$, $j\in \{0,1\}$, and
$\tilde \psi_1(\rho):=\rho^\ell \tilde \phi_1(\rho^2)$.
By construction, $\psi_j$ and $\tilde \psi_1$ are solutions to 
Eq.~\eqref{eq:ode} with $g_{\ell,m}=0$.
We have the Wronskian relation
\[ W(\rho):=
W(\psi_0,\psi_1)(\rho)=-\frac{2^{\frac12+\ell}}{\rho^2(1-\rho^2)^\frac32}. \]
This formula can be derived in a straightforward manner from the 
explicitly known connection coefficients for hypergeometric functions.
The detailed computation is given in \cite{DonZen14} on p.~477.
By applying the variation of constants formula, we obtain a solution
to Eq.~\eqref{eq:ode} given by
\begin{align}
\label{eq:vkf}
u_{\ell,m}(\rho)=&-\psi_1(\rho)\int_0^\rho \frac{\psi_0(s)}{(1-s^2)W(s)}
g_{\ell,m}(s)ds \nonumber \\
&-\psi_0(\rho)\int_\rho^1 \frac{\psi_1(s)}{(1-s^2)W(s)}
g_{\ell,m}(s)ds.
\end{align}
Eq.~\eqref{eq:vkf} has been analyzed in \cite{DonZen14} and we take
some results from there.
Since $g_{\ell,m}\in C^\infty[0,1]$, it follows that $u_{\ell,m}\in 
C^\infty(0,1)$ but the precise regularity properties at the endpoints
are more difficult to establish.
In fact, the endpoint $\rho=1$ is the nontrivial one whereas the singularity
at the center is just an artifact of the polar coordinates.
For $\rho \in (0,1)$ we may differentiate Eq.~\eqref{eq:vkf} which yields
\[ u_{\ell,m}''(\rho)=-\psi_1''(\rho)I_0(\rho)-\psi_0''(\rho)I_1(\rho)
-\frac{g_{\ell,m}(\rho)}{1-\rho^2} \]
where $I_0$ and $I_1$ stand for the first and second integral expression
in Eq.~\eqref{eq:vkf}, respectively.
There is no reason to believe that $g_{\ell,m}(1)=0$ and therefore, 
$u_{\ell,m}''(\rho)$ has an apparent singularity at $\rho=1$.
From $|\tilde \psi_1(\rho)|\simeq (1-\rho)^{-\frac12}$ as $\rho\to 1-$
it follows that $|\psi_0(\rho)|\lesssim (1-\rho)^{-\frac12}$ for
$\rho$ near $1$.
As a consequence, we have $\psi_1'' I_0 \in C^1(0,1]$. 
On the other hand, we have $|\psi_0''(\rho)I_1(\rho)|\simeq (1-\rho)^{-1}$.
In fact, even more is true. By the structure of the integrand and the explicit
form of the solutions $\psi_j$ it is clear
that we have $\psi_0''(\rho)I_1(\rho)=\frac{h(\rho)}{1-\rho^2}$ for
a function $h \in C^2(0,1]$.
Furthermore, a straightforward computation using de l'H\^opital's rule shows that
\[ \lim_{\rho\to 1-}h(\rho)=
\lim_{\rho\to 1-}[(1-\rho^2)\psi_0''(\rho)I_1(\rho)]=-g_{\ell,m}(1), \]
see \cite{DonZen14}, p.~478.
As a consequence, we infer $u_{\ell,m}\in C^3(0,1]$.
Combining this with the bounds near the center from \cite{DonZen14} we see
that the function $u$ defined in Eq.~\eqref{eq:ansatz} belongs
to $H^2(\B^3)\cap C^3(\overline{\B^3}\backslash\{0\})$.
By elliptic regularity we infer $u\in C^\infty(\B^3)\cap
C^3(\overline{\B^3}\backslash\{0\})$ which implies
$u\in C^3(\overline{\B^3})$ as desired.
\end{proof}
 
As promised, Lemma \ref{lem:degen} implies the desired density property.

\begin{lemma}
\label{lem:dense}
The operator $\tfrac32-\tfrac{2}{p-1}-\tilde{\mb L}$ has dense range.
\end{lemma}

\begin{proof}
Let $\mb f\in C^\infty(\overline{\B^3})\times C^\infty(\overline{\B^3})$ and $\varepsilon>0$. 
The equation $(\lambda-\tilde{\mb L})\mb u=\mb f$ reads
\[
\left \{ \begin{array}{l}
\xi^j \partial_j u_1(\xi)+(\lambda+\tfrac{2}{p-1}) u_1(\xi)-u_2(\xi)=f_1(\xi) \\
-\partial^j \partial_j u_1(\xi)+\xi^j \partial_j u_2(\xi)
+(\lambda+\tfrac{2}{p-1}+1) u_2(\xi)
=f_2(\xi)
\end{array} \right .
\]
and by inserting the expression for $u_2$ from the first equation into the second one,
we obtain the degenerate elliptic problem
\begin{align}
\label{eq:degen}
&-(\delta^{jk}-\xi^j \xi^k)\partial_j \partial_k u_1(\xi)
+2(\lambda+\tfrac{2}{p-1}+1)\xi^j\partial_j u_1(\xi) \nonumber \\
&+(\lambda+\tfrac{2}{p-1})(\lambda+\tfrac{2}{p-1}+1)u_1(\xi) \nonumber \\
&\quad =\xi^j\partial_j f_1(\xi)+(\lambda+\tfrac{2}{p-1}+1)f_1(\xi)+f_2(\xi).
\end{align}
Setting $\lambda=\frac32-\frac{2}{p-1}$ yields
\[ -(\delta^{jk}-\xi^j\xi^k)\partial_j \partial_k u_1(\xi)
+5\xi^j\partial_j u_1(\xi)+\tfrac{15}{4}u_1(\xi)=f(\xi) \]
where $f\in C^\infty(\overline{\B^3})$ is given by
$f(\xi)=\xi^j\partial_j f_1(\xi)+\tfrac52 f_1(\xi)+f_2(\xi)$.
From Lemma \ref{lem:degen} we infer the existence of a function $u\in C^3(\overline{\B^3})$
such that
\[ g(\xi):=-(\delta^{jk}-\xi^j\xi^k)\partial_j \partial_k u(\xi)
+5\xi^j\partial_j u(\xi)+\tfrac{15}{4}u(\xi) \]
satisfies $\|f-g\|_{H^1(\B^3)}<\varepsilon$.
Now we define $\mb u \in C^3(\overline{\B^3})\times C^2(\overline{\B^3})
=\mc D(\tilde{\mb L})$ 
and $\mb g \in C^2(\overline{\B^3})\times C^1(\overline{\B^3})$ by
\begin{align*}
u_1&:=u \\
u_2(\xi)&:=\xi^j \partial_j u(\xi)+\tfrac32 u(\xi)-f_1(\xi) \\
g_1&:=f_1 \\
g_2(\xi)&:=g(\xi)-\xi^j \partial_j f_1(\xi)-\tfrac52 f_1(\xi).
\end{align*}
By construction, we have $(\frac32-\frac{2}{p-1}-\tilde{\mb L})\mb u=\mb g$
and
\[ \|\mb f-\mb g\|=\|f_2-g_2\|_{H^1(\B^3)}=\|f-g\|_{H^1(\B^3)}<\varepsilon. \]
Consequently, the claim follows from the density 
of $C^\infty(\overline{\B^3})\times 
C^\infty(\overline{\B^3})$ in $\mc H$.
\end{proof}

As a culmination of our efforts we can now establish the desired generation
result.

\begin{proof}[Proof of Proposition \ref{prop:free}]
From Lemmas \ref{lem:sesq} and \ref{lem:sesq3} we have the estimate
\[ \Re(\tilde{\mb L}\mb u|\mb u)\leq -\tfrac{2}{p-1}\|\mb u\|^2 \]
for all $\mb u\in \mc D(\tilde{\mb L})$ 
and by Lemma \ref{lem:dense}, the range of $\frac32-\frac{2}{p-1}-\tilde{\mb L}$
is dense.
Thus, the claim follows from the Lumer-Phillips Theorem \cite{EngNag00}, p.~83, Theorem 3.15.
\end{proof}

\subsection{The modulation ansatz}
Based on Proposition \ref{prop:free}, Eq.~\eqref{eq:syspsi} may be written in abstract
form as
\begin{equation}
\label{eq:Psi}
 \partial_\tau \Psi(\tau)=\mb L \Psi(\tau)+\mb N(\Psi(\tau)) 
 \end{equation}
for a function $\Psi: I \to \mc H$, $I\subset [0,\infty)$ an interval with $0\in I$,
and
\[ \mb N(\mb u):=\left (\begin{array}{c} 0 \\ u_1|u_1|^{p-1} \end{array} \right ). \]
Our goal is to study the stability of the explicit 3-parameter family of static solutions
\[ \Psi_a:=\left (\begin{array}{c}\psi_{a,1} \\ \psi_{a,2}\end{array} \right ) \]
with $\psi_{a,1}$, $\psi_{a,2}$ from Eq.~\eqref{eq:syspsi_a}.
The standard modulation ansatz consists of looking 
for solutions to Eq.~\eqref{eq:Psi} of the form
\begin{equation}
\label{eq:modansatz} \Psi(\tau)=\Psi_{a(\tau)}+\Phi(\tau). 
\end{equation}
We assume that $a_\infty:=\lim_{\tau\to \infty}a(\tau)$ exists and
insert the ansatz \eqref{eq:modansatz} into Eq.~\eqref{eq:Psi}.
This yields the evolution equation
\begin{align} 
\label{eq:Phi}
\partial_\tau \Phi(\tau)-\mb L \Phi(\tau)-\mb L_{a_\infty}' \Phi(\tau)=
&[\mb L'_{a(\tau)}-\mb L_{a_\infty}']\Phi(\tau) \nonumber \\
&+\mb N_{a(\tau)}(\Phi(\tau))-\partial_\tau \Psi_{a(\tau)} 
\end{align}
where
\begin{align*} \mb L'_{a(\tau)}\mb u(\xi)&=\left (\begin{array}{c} 0 \\
p\psi_{a(\tau),1}^{p-1}u_1(\xi) \end{array} \right ) \\
&=\frac{2p(p+1)}{(p-1)^2 [A_0(a(\tau))-A_j(a(\tau))\xi^j]^{2}}\left (\begin{array}{c} 0 \\
u_1(\xi) \end{array} \right ) 
\end{align*}
and 
\[ \mb N_{a(\tau)}(\mb u):=\mb N(\Psi_{a(\tau)}+\mb u)-\mb N(\Psi_{a(\tau)})-
\mb L'_{a(\tau)}\mb u. \]
The right-hand side of Eq.~\eqref{eq:Phi} will be treated perturbatively.
We note the following convenient property of $\mb L'_a$ which in particular yields
the existence of a semigroup that governs the linear evolution.

\begin{lemma}
\label{lem:compact}
Let $a=(a^1,a^2,a^3)\in \R^3$ be sufficiently small.
Then the operator $\mb L_a'$ is compact.
In particular, $\mb L_a:=\mb L+\mb L_a'$ generates
a strongly-continuous one-parameter semigroup $\mb S_a: [0,\infty)\to \mc B(\mc H)$.
\end{lemma}

\begin{proof}
If $a$ is not too large, we have 
$\psi_{1,a}^{p-1}\in C^\infty(\overline{\B^3})$.
Since $\mb L'_a$ maps the first component of a function
$\mb u\in \mc H=H^2(\B^3)\times H^1(\B^3)$
to the second row, the compactness of $H^2(\B^3)\hookrightarrow H^1(\B^3)$ shows
that $\mb L_a'$ is compact.
The existence of $\mb S_a$ 
follows from the Bounded Perturbation Theorem.
\end{proof}

\section{Spectral analysis}

\noindent Lemma \ref{lem:compact} yields the existence of the semigroup $\mb S_a$, 
but establishing 
a useful growth estimate for $\mb S_a$ is nontrivial and requires a detailed
spectral analysis of the generator $\mb L_a$.
The fact that $\mb L_a$ depends on a parameter complicates matters even further.
However, we are only interested in small $a$ which allows for a perturbative approach.

\subsection{The spectrum of $\mb L_0$}
We start the analysis with the case $a=0$ where the spectral equation can be 
solved explicitly, as it turns out.
First, however, we observe the general fact that every spectral point of $\mb L_a$ outside
of $\sigma(\mb L)$ is an eigenvalue.

\begin{lemma}
\label{lem:point}
Let $a\in \R^3$ be sufficiently small. If $\lambda \in \sigma(\mb L_a)\backslash
\sigma(\mb L)$ then $\lambda\in \sigma_p(\mb L_a)$.
\end{lemma}

\begin{proof}
We use the identity $\lambda-\mb L_a=
[1-\mb L'_a\mb R_{\mb L}(\lambda)](\lambda-\mb L)$ to see that
$\lambda\in \sigma(\mb L_a)$ implies
$1\in \sigma(\mb L'_a \mb R_{\mb L}(\lambda))$.
By the compactness of $\mb L'_a \mb R_{\mb L}(\lambda)$ (Lemma \ref{lem:compact})
we obtain
$1\in \sigma_p(\mb L'_a \mb R_{\mb L}(\lambda))$.
Let $\mb u$ be a corresponding eigenfunction.
Then $\mb R_{\mb L}(\lambda)\mb u$ is an eigenfunction of $\mb L_a$
with eigenvalue $\lambda$
and we obtain $\lambda\in \sigma_p(\mb L_a)$ as claimed.
\end{proof}

Now we can give a sufficiently detailed description of the spectrum of $\mb L_0$.

\begin{lemma}
\label{lem:spec0}
We have
\[ \sigma(\mb L_0) \subset \{z\in \C: \Re z\leq -\tfrac{2}{p-1}\}\cup \{0,1\} \]
and $\{0,1\}\subset \sigma_p(\mb L_0)$.
Furthermore, the geometric eigenspaces of the eigenvalues $1$ and $0$ are spanned
by the functions $\mb g_0$ and $\mb h_{0,j}$, respectively, where
\begin{align*} 
\mb g_0(\xi)&=\left (\begin{array}{c}1 \\ \frac{p+1}{p-1} \end{array} \right ) \\
\mb h_{0,j}(\xi)&=\left. \partial_{a^j}\Psi_a(\xi)\right |_{a=0}=
\tfrac{2c_p}{p-1}\left (\begin{array}{c}\xi_j \\ \frac{p+1}{p-1}\xi_j \end{array} \right ).
\end{align*}
\end{lemma}

\begin{proof}
The growth bound $\|\mb S(\tau)\|\leq e^{-\frac{2}{p-1}\tau}$ from 
Proposition \ref{prop:free} implies
$\sigma(\mb L)\subset \{z\in \C: \Re z\leq -\tfrac{2}{p-1}\}$.
Thus, by Lemma \ref{lem:point} it suffices
to look for eigenvalues of $\mb L_0$.
The eigenvalue equation $(\lambda-\mb L_0)\mb u=0$ reduces to
\begin{align}
\label{eq:spec}
-(\delta^{jk}&-\xi^j \xi^k)\partial_j \partial_k u_1(\xi)
+2(\lambda+\tfrac{2}{p-1}+1)\xi^j\partial_j u_1(\xi) \nonumber \\
&+(\lambda+\tfrac{2}{p-1})(\lambda+\tfrac{2}{p-1}+1)u_1(\xi)-p\psi_{0,1}^{p-1}u_1(\xi)=0 
\end{align}
and $u_2(\xi)=\xi^j\partial_j u_1(\xi)+(\lambda+\tfrac{2}{p-1})u_1(\xi)$, see
Eq.~\eqref{eq:degen}.
Thus, it suffices to solve Eq.~\eqref{eq:spec}.
Recall from Eq.~\eqref{eq:syspsi_a} that $p\psi_{0,1}^{p-1}(\xi)=\frac{2p(p+1)}{(p-1)^2}$, i.e.,
it is in fact a constant.
If $\mb u \in \mc D(\mb L)$ satisfies $(\lambda-\mb L_0)\mb u=0$ 
then $u_1\in H^2(\B^3)$ and elliptic regularity applied to Eq.~\eqref{eq:spec}
implies $u_1\in C^\infty(\B^3)\cap H^2(\B^3)$.
Consequently, by introducing polar coordinates $\xi=\rho\omega$, 
$(\rho,\omega)\in (0,1]\times \S^2$, we may expand $u_1$ as
\[ u_1(\rho\omega)=\sum_{\ell=0}^\infty \sum_{m=-\ell}^\ell u_{\ell,m}(\rho)Y_{\ell,m}(\omega) \]
with $u_{\ell,m}(\rho)=(u_1(\rho\,\cdot)|Y_{\ell,m})_{\S^2}$ and for any $\delta>0$,
the sum converges in $H^k(\B^3_{1-\delta})$ for arbitrary $k\in \N_0$, see
Appendix \ref{sec:apx}. 
Since the radial derivative and the Laplace-Beltrami operator commute with the
sum, Eq.~\eqref{eq:spec} is equivalent to
the decoupled system of ODEs
\begin{align}
\label{eq:specode}
\big [-(1&-\rho^2)\partial_\rho^2-\tfrac{2}{\rho}\partial_\rho
+2(\lambda+\tfrac{2}{p-1}+1)\rho\partial_\rho+\tfrac{\ell(\ell+1)}{\rho^2} \nonumber \\
&+(\lambda+\tfrac{2}{p-1})(\lambda+\tfrac{2}{p-1}+1)
-\tfrac{2p(p+1)}{(p-1)^2}\big ]u_{\ell,m}(\rho)=0
\end{align}
and $u_1 \in C^\infty(\B^3)\cap H^2(\B^3)$ implies $u_{\ell,m}\in C^\infty[0,1)
\cap H^2(\frac12,1)$.
As in the proof of Lemma \ref{lem:degen} we suppress the subscripts and
set $u_{\ell,m}(\rho)=\rho^\ell v(\rho^2)$.
Then $u_{\ell,m}$ satisfies Eq.~\eqref{eq:specode} iff $v$ solves the hypergeometric
differential equation
\begin{equation}
\label{eq:hypgeo}
 z(1-z)v''(z)+[c-(a+b+1)z]v'(z)-abv(z)=0 
 \end{equation}
where $z=\rho^2$ and $a=\frac{\lambda}{2}-\frac12+\frac{\ell}{2}$,
$b=\frac{\lambda}{2}+\frac{p+1}{p-1}+\frac{\ell}{2}$, $c=\frac32+\ell$.
Note that $u_{\ell,m}\in H^2(\frac12,1)$ implies $v\in H^2(\frac12,1)$.
We consider the solutions
\begin{align*}
v_0(z)&:={}_2F_1(a,b,c;z) \\
\tilde v_0(z)&:=z^{1-c}{}_2F_1(a-c+1, b-c+1, 2-c; z) \\
v_1(z)&:={}_2F_1(a,b,a+b+1-c;1-z) \\
\tilde v_1(z)&:=(1-z)^{c-a-b}{}_2F_1(c-a, c-b, c-a-b+1; 1-z)
\end{align*}
with ${}_2F_1$ the standard hypergeometric function (see \cite{DLMF}).
Since we are only interested in $\Re\lambda>-\tfrac{2}{p-1}$, we obtain
\[ \Re(c-a-b)=1-\Re \lambda-\tfrac{2}{p-1}<1 \] 
and thus, the condition $v \in H^2(\frac12,1)$ 
excludes\footnote{To be honest, if $c-a-b=0$ one needs a slightly different
argument. In this pathological
case the ``bad'' solution at $z=1$ is not given by $\tilde v_1$ but it
behaves like a logarithm. However, this leads to the same conclusion since $z\mapsto
\log(1-z)$
is not in $H^2(\frac12,1)$ either.} the solution $\tilde v_1$.
Similarly, the solution $\tilde v_0$ is not admissible either since it would lead
to a $u_{\ell,m}$ that behaves like $\rho^{-1-\ell}$ as $\rho\to 0+$ which contradicts
$u_{\ell,m}\in C^\infty[0,1)$.
As a consequence, since both $\{v_0,\tilde v_0\}$ and $\{v_1,\tilde v_1\}$ are
fundamental systems for Eq.~\eqref{eq:hypgeo}, 
we see that $v_0$ and $v_1$ must be linearly dependent.
In view of the connection formula \cite{DLMF}
\[ v_0(z)=\frac{\Gamma(c)\Gamma(c-a-b)}{\Gamma(c-a)\Gamma(c-b)}v_1(z)
+\frac{\Gamma(c)\Gamma(a+b-c)}{\Gamma(a)\Gamma(b)}\tilde v_1(z) \]
this is only possible if $\frac{\Gamma(c)\Gamma(a+b-c)}{\Gamma(a)\Gamma(b)}=0$.
Since the $\Gamma$-function does not have zeros, we see that $a$ or $b$ must be
a pole of $\Gamma$, i.e., we obtain $-a\in \N_0$ or $-b\in \N_0$.
The latter condition translates into $\lambda=-\frac{2}{p-1}-\frac{2p}{p-1}-\ell-2n$
for some $n\in \N_0$ but this cannot hold for any $n\in \N_0$
since we assume $\Re\lambda>-\frac{2}{p-1}$.
Consequently, we are left with $-a\in \N_0$ which means that
$\lambda=1-\ell-2n$
for some $n\in \N_0$.
This is compatible with $\Re\lambda>-\frac{2}{p-1}>-1$ only if
$n=0$ and $\ell \in \{0,1\}$, i.e., $\lambda \in \{0,1\}$.
This shows $\sigma(\mb L_0)\subset \{z\in \C: \Re z\leq -\frac{2}{p-1}\}\cup \{0,1\}$.

Moreover, a straightforward computation shows that $\mb g_0$ and $\mb h_{0,j}$ are
eigenfunctions of $\mb L_0$ with eigenvalues $1$ and $0$, respectively, which implies
$\{0,1\}\subset \sigma_p(\mb L_0)$.
Finally, the above derivation also shows that the geometric eigenspaces of
$1$ and $0$ are at most $1$-dimensional and $3$-dimensional, respectively.
\end{proof}

Now we define the usual Riesz projections associated to the eigenvalues
$0$ and $1$ of $\mb L_0$.
We set
\begin{align*}
\mb P_0&:=\frac{1}{2\pi \I}\int_{\gamma_0}\mb R_{\mb L_0}(z)dz \\
\mb Q_0&:=\frac{1}{2\pi \I}\int_{\gamma_1}\mb R_{\mb L_0}(z)dz
\end{align*}
with the curves $\gamma_j: [0,1]\to \C$, $j\in \{0,1\}$, given by
\[ \gamma_0(s)=\tfrac{1}{p-1}e^{2\pi \I s},\qquad
\gamma_1(s)=1+\tfrac{1}{2}e^{2\pi \I s}. \]

The following important result shows that the algebraic multiplicities of the eigenvalues
$1$ and $0$ equal their geometric multiplicities. 
Note that this is a nontrivial fact since we are in a highly nonself-adjoint setting.

\begin{lemma}
\label{lem:alg0}
The projections $\mb P_0$ and $\mb Q_0$ have rank $1$ and $3$, respectively.
\end{lemma}

\begin{proof}
First of all we note that $\mb P_0$ and $\mb Q_0$ have finite
rank. This is because the eigenvalues $0$ and $1$ are generated by a compact perturbation
and an eigenvalue of infinite algebraic multiplicity is invariant under such
a perturbation (see \cite{Kat95}, p.~ 239, Theorem 5.28 and p.~244, Theorem 5.35).

Next, observe that $\rg \mb Q_0 \subset \mc D(\mb L)$.
To see this, let $\mb v\in \rg \mb Q_0$.
By the density of $\mc D(\mb L)$ in $\mc H$ 
we find $(\mb u_n)\subset \mc D(\mb L)$ with
$\mb u_n \to \mb v$.
Since $\mb Q_0 \mc D(\mb L)\subset \mc D(\mb L)$ (\cite{Kat95}, p.~178, Theorem 6.17)
we see that $(\mb v_n):=(\mb Q_0 \mb u_n)\subset
\mc D(\mb L)\cap \rg \mb Q_0$ and $\mb v_n=\mb Q_0 \mb u_n \to \mb Q_0 \mb v=\mb v$ 
by the boundedness of $\mb Q_0$.
The operator $\mb L_0|_{\mc D(\mb L)\cap \rg \mb Q_0}$ 
is bounded and this implies $\mb L_0 \mb v_n\to \mb f$
for some $\mb f\in \rg \mb Q_0$.
Thus, the closedness of $\mb L_0$ implies $\mb v\in \mc D(\mb L)$ and we obtain
$\rg \mb Q_0\subset \mc D(\mb L)$ as claimed.

$\mb A:=\mb L_0|_{\rg \mb Q_0}$ is a bounded operator
on the finite-dimensional Hilbert space $\rg \mb Q_0$ with spectrum equal to 
$\{0\}$ (\cite{Kat95}, p.~178, Theorem 6.17).
This implies that
$\mb A$ is nilpotent, i.e., there exists a (minimal) $n\in \N$
such that $\mb A^n=\mb 0$.
If $n=1$ we have $\rg \mb Q_0=\ker \mb A$ and any element of $\rg \mb Q_0$
is an eigenvector of $\mb A$ (hence $\mb L_0$) with eigenvalue $0$.
From Lemma \ref{lem:spec0} we infer $\rg \mb Q_0=\langle \mb h_{0,1},\mb h_{0,2},
\mb h_{0,3}\rangle$ which shows that $\mb Q_0$ has rank $3$.
Now suppose $n\geq 2$. 
Then there exists a $\mb u\in \rg \mb Q_0 \subset \mc D(\mb L)$ 
such that $\mb A\mb u$ is a nontrivial element of $\ker \mb A\subset \ker \mb L_0$, i.e.,
we have $\mb L_0 \mb u=\mb f$ for some nonzero $\mb f\in \ker \mb L_0$.
We obtain the equation
 \begin{align}
-(\delta^{jk}&-\xi^j \xi^k)\partial_j \partial_k u_1(\xi)
+2(\tfrac{2}{p-1}+1)\xi^j\partial_j u_1(\xi) \nonumber \\
&+\tfrac{2}{p-1}(\tfrac{2}{p-1}+1)u_1(\xi)-\tfrac{2p(p+1)}{(p-1)^2} u_1(\xi)=
f(\xi) 
\end{align}
with $f(\xi)=\xi^j \partial_j f_1(\xi)+(\frac{2}{p-1}+1)f_1(\xi)+f_2(\xi)$, cf.~Eq.~\eqref{eq:degen}.
From Lemma \ref{lem:spec0} we infer that $f$ is of the form 
\[ f(\xi)=\tilde \alpha_j \xi^j
=|\xi|\sum_{m=-1}^1 \alpha_m Y_{1,m}(\tfrac{\xi}{|\xi|}) \]
where $\alpha_m\not=0$ for at least one $m\in \{-1,0,1\}$
and without loss of generality we may assume $\alpha_0=1$.
Consequently, an angular momentum decomposition as in the proof of Lemma \ref{lem:spec0}
leads to the inhomogeneous ODE
\begin{align}
\label{eq:odeinh}
\big [-(1&-\rho^2)\partial_\rho^2-\tfrac{2}{\rho}\partial_\rho
+2(\tfrac{2}{p-1}+1)\rho\partial_\rho+\tfrac{2}{\rho^2} \nonumber \\
&+\tfrac{2}{p-1}(\tfrac{2}{p-1}+1)
-\tfrac{2p(p+1)}{(p-1)^2}\big ]u_{1,0}(\rho)=\rho
\end{align}
for the function $u_{1,0}(\rho)=(u_1(\rho\,\cdot)|Y_{1,0})_{\S^2}$
and we have $u_{1,0}\in C^\infty(0,1)\cap H^2_\mathrm{rad}(0,1)$.
The homogeneous version of Eq.~\eqref{eq:odeinh} has the solution
$\phi(\rho)=\rho$ and the usual reduction ansatz immediately yields a second
solution $\psi$ given by\footnote{Recall that we assume $p>3$, i.e., $\frac{2}{p-1}<1$.}
\[ \psi(\rho)=-\rho \int_\rho^1 \frac{ds}{s^4(1-s^2)^\frac{2}{p-1}}. \]
By construction, we have $W(\rho):=W(\phi,\psi)(\rho)=\rho^{-2}(1-\rho^2)^{-\frac{2}{p-1}}$
and the variation of constants formula shows that $u_{1,0}$ must be of the form
\begin{align*} u_{1,0}(\rho)=&c_0 \phi(\rho)+
c_1 \psi(\rho)-\psi(\rho)\int_{\rho_1}^\rho \frac{\phi(s)}{W(s)(1-s^2)}sds \\
&+\phi(\rho)\int_{\rho_0}^\rho
\frac{\psi(s)}{W(s)(1-s^2)}sds 
\end{align*}
for constants $c_0,c_1 \in \C$ and $\rho_0,\rho_1\in [0,1]$.
We have $|\psi(\rho)|\simeq \rho^{-2}$ as $\rho \to 0+$ and thus, the property
$u_{1,0}\in L^2_\mathrm{rad}(0,1)$ shows that we must have $c_1=\int_{\rho_1}^0 \frac{\phi(s)}
{W(s)(1-s^2)}sds$ which leaves us with
\begin{align*} u_{1,0}(\rho)=&c_0 \phi(\rho)-\psi(\rho)
\int_0^\rho \frac{\phi(s)}{W(s)(1-s^2)}sds \\
&+\phi(\rho)\int_{\rho_0}^\rho
\frac{\psi(s)}{W(s)(1-s^2)}sds. 
\end{align*}
Furthermore, since $\psi(\rho)=(1-\rho)^{1-\frac{2}{p-1}}h(\rho)$ 
for some function $h$ which is smooth at $1$ and $h(1)\not=0$, we see that 
$\rho\mapsto \frac{\psi(\rho)}{W(\rho)(1-\rho^2)}$ is also smooth at $\rho=1$.
Consequently, the property $u_{1,0}\in H^2(\frac12,1)$ implies
\[ \int_0^1 \frac{\phi(s)}{W(s)(1-s^2)}sds=0 \]
but this is impossible since the integrand is positive on $(0,1)$.
This contradiction shows that we must have $n=1$ and thus, $\mb Q_0$ has rank $3$.
By the exact same argument one proves that $\mb P_0$ has rank $1$.
\end{proof}

\subsection{The spectrum of $\mb L_a$}
Now that we have understood the spectrum of $\mb L_0$, we turn to $\mb L_a$ for
$a\not= 0$.
We are only interested in small $a$ and thus, the problem can be treated perturbatively.
A first and crucial observation in this respect is the fact that
$\mb L'_a$ depends continuously on $a$.

\begin{lemma}
\label{lem:LipL'}
There exists a $\delta>0$ such that
\begin{align*} 
\|\mb L_a'-\mb L_b'\|&\lesssim |a-b| 
\end{align*}
for all $a,b\in \overline{\B^3_{\delta}}$.
\end{lemma}
 
 \begin{proof}
 Let $\mb u\in C^\infty(\overline{\B^3})^2$. The fundamental theorem of calculus yields the
pointwise representation
\[ \mb L'_b \mb u(\xi)-\mb L'_a \mb u(\xi)
=\left (\begin{array}{c}0 \\ u_1(\xi) \end{array} \right )(b^j-a^j)\int_0^1 \varphi_{a+s(b-a),j}(\xi)ds
 \]
 for $a,b \in\B^3_\delta$ 
 and $\varphi_{a,j}(\xi):=p\partial_{a^j}\psi_{a,1}^{p-1}(\xi)$,
provided $\delta>0$ is not too large.
The function $(\xi,a,b,s)\mapsto \varphi_{a+s(b-a),j}(\xi)$ belongs to 
$C^\infty(\overline{\B^3}\times \overline{\B^3_\delta}
\times \overline{\B^3_\delta} \times [0,1])$
 and the claimed Lipschitz bound for $\mb L'_a$ follows.
 \end{proof}
 
 As a corollary we infer that the spectrum of $\mb L_a$ in compact domains
 does not differ too much
 from the spectrum of $\mb L_0$.
 
 \begin{corollary}
 \label{cor:resLa}
 There exists a $\delta>0$ such that $\lambda\in \rho(\mb L_0)$ implies
 $\lambda\in \rho(\mb L_a)$ provided
$|a|\leq \delta \min \left \{1, \|\mb R_{\mb L_0}(\lambda)\|^{-1} \right \}$.
 \end{corollary}
 
 \begin{proof}
 Let $\delta_1>0$ be so small that $\mb L_a'$ is Lipschitz-continuous 
 for all $a\in \overline{\B^3_{\delta_1}}$ (Lemma \ref{lem:LipL'}).
 Now assume $\lambda\in \rho(\mb L_0)$.
 Then, for $a\in \overline{\B^3_{\delta_1}}$, we may use the identity 
 $\lambda-\mb L_a=[1+(\mb L_0'-\mb L_a')\mb R_{\mb L_0}(\lambda)](\lambda-\mb L_0)$
 to see that $\lambda-\mb L_a$ is bounded invertible precisely if
 $1+(\mb L_0'-\mb L_a')\mb R_{\mb L_0}(\lambda)$ is bounded invertible.
 By a Neumann series argument it follows that this is certainly the case if
 $\|\mb L_0'-\mb L_a'\|\|\mb R_{\mb L_0}(\lambda)\|<1$.
 From Lemma \ref{lem:LipL'} we know that there exists a constant $L>0$ such that
 $\|\mb L_0'-\mb L_a'\|\leq L|a|$.
 Consequently, if we choose $\delta=\min\{\delta_1,\frac{1}{2L}\}$, we see that
 $|a|\leq \delta \min\{1,\|\mb R_{\mb L_0}(\lambda)\|^{-1}\}$ implies $a\in \overline{\B^3_{\delta_1}}$ and 
 \[ \|\mb L_0'-\mb L_a'\|\|\mb R_{\mb L_0}(\lambda)\|\leq L|a|\|\mb R_{\mb L_0}(\lambda)\|
 <1 \] 
which yields $\lambda\in \rho(\mb L_a)$.
 \end{proof}

Now we define 
\[ \Omega_{\kappa_0,\omega_0}:=
\left \{z\in \C: \Re z \in [-\tfrac{3/2}{p-1}, \kappa_0], \Im z \in [-\omega_0,\omega_0]
\right \} \]
and 
\[ \Omega_{\kappa_0,\omega_0}':=\{z\in \C: 
\Re z\geq -\tfrac{3/2}{p-1}\} \backslash \Omega_{\kappa_0,\omega_0}. \]
The next estimate confines possible unstable eigenvalues of 
$\mb L_a$ to a compact domain $\Omega_{\kappa_0,\omega_0}$.

\begin{lemma}
\label{lem:Res}
There exist $\kappa_0,\omega_0,c,\delta>0$ such that $\Omega'_{\kappa_0,\omega_0}\subset 
\rho(\mb L_a)$ and we have
\[ \|\mb R_{\mb L_a}(\lambda)\|\leq c \]
for all $\lambda \in \Omega'_{\kappa_0,\omega_0}$ and all $a\in \overline{\B^3_\delta}$.
\end{lemma}

\begin{proof}
Let $\delta>0$ be from Lemma \ref{lem:LipL'}.
For any $\kappa_0,\omega_0>0$, the condition $\lambda\in \Omega'_{\kappa_0,\omega_0}$
implies $\lambda \in \rho(\mb L)$.
Thus, for $a\in \overline{\B^3_\delta}$ we may use the identity 
$\lambda-\mb L_a=[1-\mb L_a' \mb R_{\mb L}(\lambda)](\lambda-\mb L)$
to relate $\mb R_{\mb L_a}(\lambda)$ to the free resolvent.
Note that $\lambda \in \Omega'_{\kappa_0,\omega_0}$ implies $|\lambda|$ large
if $\kappa_0$ and $\omega_0$ are large.
Consequently, in view of the Neumann series it suffices to prove that
\[ \|\mb L_a' \mb R_{\mb L}(\lambda)\|\lesssim |\lambda|^{-1} \]
for all $\lambda\in \Omega'_{10,10}$ and all $a\in \overline{\B^3_\delta}$.
Let $\mb f\in \mc H$ and set $\mb u=\mb R_{\mb L}(\lambda)\mb f$.
Then $\mb u\in \mc D(\mb L)$ and $(\lambda-\mb L)\mb u=\mb f$.
As in the proof of Lemma \ref{lem:dense} the latter equation implies
\[ \xi^j \partial_j u_1(\xi)+(\lambda+\tfrac{2}{p-1}) u_1(\xi)-u_2(\xi)=f_1(\xi) \]
which yields the bound
\[ \|u_1\|_{H^1(\B^3)}\lesssim \frac{1}{|\lambda|}\left (
\|u_1\|_{H^2(\B^3)}+\|u_2\|_{H^1(\B^3)}+\|f_1\|_{H^1(\B^3)} \right ). \]
By inserting the definition of $\mb u$ this yields
\begin{align*} 
\|[\mb R_{\mb L}(\lambda)\mb f]_1\|_{H^1(\B^3)}&\lesssim |\lambda|^{-1}\left (
\|\mb R_{\mb L}(\lambda)\mb f\|+\|\mb f\|\right ) \\
&\lesssim |\lambda|^{-1}\|\mb f\|
\end{align*}
for all $\lambda \in \Omega'_{10,10}$ since $\|\mb R_{\mb L}(\lambda)\|\leq \frac{1}{\Re\lambda
+\frac{2}{p-1}}$ by Proposition \ref{prop:free}.
Consequently, by the definition of $\mb L_a'$ we obtain
\[ \|\mb L_a' \mb R_{\mb L}(\lambda)\mb f\|\lesssim \|[\mb R_{\mb L}(\lambda)\mb f]_1
\|_{H^1(\B^3)}\lesssim |\lambda|^{-1}\|\mb f\| \]
for all $a\in \overline{\B^3_\delta}$ and all $\lambda\in \Omega'_{10,10}$.
\end{proof}

Based on the above, we can now completely describe the spectrum of $\mb L_a$ 
in the half-space $\{z\in \C: \Re z\geq -\frac{3/2}{p-1}\}$.

\begin{lemma}
\label{lem:speca}
Let $a\in \R^3$ be sufficiently small.
Then we have
\[ \sigma(\mb L_a)\subset \{z\in \C: \Re z<-\tfrac{3/2}{p-1}\}\cup \{0,1\} \]
and $\{0,1\}\subset \sigma_p(\mb L_a)$.
Furthermore, the eigenspaces associated to the eigenvalues $1$ and $0$ are spanned by
$\mb g_a$ and $\mb h_{a,j}$, respectively, where
\begin{align*} 
\mb g_a(\xi)&=\left ( \begin{array}{c}
A_0(a)[A_0(a)-A_j(a)\xi^j]^{-\frac{2}{p-1}-1} \\
\frac{p+1}{p-1}A_0(a)^2[A_0(a)-A_j(a)\xi^j]^{-\frac{2}{p-1}-2}\end{array} \right )\\
\mb h_{a,j}(\xi)&=\partial_{a^j}\Psi_a(\xi).
\end{align*}
Finally, the algebraic multiplicities of the eigenvalues $1$ and $0$ are equal to
$1$ and $3$, 
respectively.
\end{lemma}

\begin{proof}
Let $\kappa_0, \omega_0>0$ be so large that $\overline{
\Omega_{\kappa_0,\omega_0}'}\subset \rho(\mb L_a)$
(Lemma \ref{lem:Res}).
Furthermore, set $M:=\max\{1,\sup_{z\in \partial\Omega_{\kappa_0,\omega_0}}
\|\mb R_{\mb L_0}(z)\|\}$.
By Corollary \ref{cor:resLa} there exists a $\delta>0$ such that
$\partial\Omega_{\kappa_0,\omega_0}\subset \rho(\mb L_a)$ for all $|a|\leq \frac{\delta}{M}$.
Now we define a projection $\mb P^\mathrm{tot}_a$ by
\[ \mb P_a^\mathrm{tot}=\frac{1}{2\pi \I}\int_{\partial\Omega_{\kappa_0,\omega_0}}
\mb R_{\mb L_a}(z)dz. \]
Note that 
$\mb R_{\mb L_a}(z)=\mb R_{\mb L}(z)[1-\mb L'_a\mb R_{\mb L}(z)]^{-1}$
and thus, the projection $\mb P_a^\mathrm{tot}$ depends continuously on $a$ 
by Lemma \ref{lem:LipL'}.
Lemma \ref{lem:alg0} shows that $\mb P_0^\mathrm{tot}$ has rank $4$ and
from \cite{Kat95}, p.~34, Lemma 4.10 
it follows that $\mb P_a^\mathrm{tot}$ has rank $4$ for all
small $a$.
On the other hand, a straightforward computation shows that $\mb g_a$ and $\mb h_{a,j}$ are 
eigenfunctions of $\mb L_a$ with eigenvalues $1$ and $0$, respectively.
The total geometric multiplicity of these eigenvalues equals $4$ and since $\mb P_a^\mathrm{tot}$
has rank $4$,
there can be no eigenvalues other than $1$ and $0$ in $\Omega_{\kappa_0,\omega_0}$.
In addition, the algebraic multiplicity of the eigenvalues $1$ and $0$ must be $1$ and $3$, 
respectively.
\end{proof}

\subsection{The linearized evolution}

The above spectral analysis leads to a sufficiently complete description of the linearized
evolution.

\begin{proposition}
\label{prop:linear}
Let $a=(a^1,a^2,a^3)\in \R^3$ be small.
Then there exist
rank-one projections $\mb P_{a}, \mb Q_{a,j} \in \mc B(\mc H)$, $j\in \{1,2,3\}$, 
satisfying 
\[ [\mb S_a(\tau),\mb P_{a}]=[\mb S_a(\tau),\mb Q_{a,j}]=\mb 0 \]
for all $\tau\geq 0$ such that
\begin{align*}
\mb S_a(\tau)\mb P_a&=e^\tau \mb P_a, \\
\mb S_a(\tau)\mb Q_{a,j}&=\mb Q_{a,j},\qquad j\in \{1,2,3\}, \\
\|\mb S_a(\tau)\tilde{\mb P}_a\mb f \|
&\lesssim e^{-\frac{4/3}{p-1}\tau}
\|\tilde{\mb P}_a\mb f\|
\end{align*}
for all $\tau\geq 0$, $\mb f\in \mc H$, where
$\tilde{\mb P}_a:=\mb I-\mb P_a-\sum_{j=1}^3 \mb Q_{a,j}$.
Furthermore, we have
\begin{align*} 
\rg \mb P_a&=\langle \mb g_a \rangle \\
\rg \mb Q_{a,j}&=\langle \mb h_{a,j} \rangle,\qquad j\in \{1,2,3\} 
\end{align*}
where $\mb g_a \in \mc H$ is an eigenfunction of $\mb L_a$ with eigenvalue $1$
and 
\[ \mb h_{a,j}(\xi):=\partial_{a^j}\Psi_a(\xi) \] 
are eigenfunctions of $\mb L_a$ with eigenvalues $0$.
Finally, $\mb Q_{a,j}\mb Q_{a,k}=\delta_{jk}\mb Q_{a,j}$ and 
$\mb Q_{a,j}\mb P_a=\mb P_a\mb Q_{a,j}=\mb 0$.
\end{proposition}

\begin{proof}
By Lemma \ref{lem:speca} we may define the spectral projections $\mb P_a$
and $\mb Q_a$ by
\[ \mb P_a:=\frac{1}{2\pi \I}\int_{\gamma_0}\mb R_{\mb L_a}(z)dz,
\qquad \mb Q_a:=\frac{1}{2\pi \I}\int_{\gamma_1}\mb R_{\mb L_a}(z)dz \]
where $\gamma_0(s)=\frac{1}{p-1}e^{2\pi \I s}$ and $\gamma_1(s)=1+\frac12 e^{2\pi\I s}$ (this
is consistent with the earlier definition of $\mb P_0$ and $\mb Q_0$).
Note that we have $\mb P_a\mb Q_a=\mb Q_a\mb P_a=\mb 0$
and $[\mb S_a(\tau),\mb P_a]=[\mb S_a(\tau),\mb Q_a]=\mb 0$.
Since $\rg \mb Q_a=\langle \mb h_{a,1},\mb h_{a,2},\mb h_{a,3}\rangle$,
we may define $\mb Q_{a,j} \mb f:=\alpha_j \mb h_{a,j}$
where $\alpha_j\in \C$ is the unique expansion coefficient in
$\mb Q_a\mb f=\sum_{k=1}^3 \alpha_k \mb h_{a,k}$.
Then we have $\mb Q_{a,j}\mb Q_{a,k}=\delta_{jk}\mb Q_{a,j}$
and $\mb Q_{a,j}\mb P_a=\mb P_a\mb Q_{a,j}=\mb 0$.
Furthermore, we obtain 
\begin{align*} \mb Q_{a,j}\mb S_a(\tau)\mb f&=\mb Q_{a,j}\mb Q_a\mb S_a(\tau)\mb f
=\mb Q_{a,j}\mb S_a(\tau)\mb Q_a \mb f=\mb Q_{a,j}\mb Q_a \mb f=\mb Q_{a,j}\mb f \\
&=\mb S_a(\tau)\mb Q_{a,j}\mb f,
\end{align*}
i.e., $[\mb S_a(\tau),\mb Q_{a,j}]=\mb 0$.
Now we set $\mathbb{H}^+_p:=\{z\in \C: \Re z\geq -\tfrac{3/2}{p-1}\}$.
From Lemmas \ref{lem:speca} and \ref{lem:Res} we obtain
$\sup_{z\in \mathbb H_p^+}\|\mb R_{\mb L_a}(z)\tilde{\mb P}\|<\infty$
and thus, the Gearhart-Pr\"u\ss{} Theorem (\cite{EngNag00}, p.~302, Theorem 1.11) 
yields the stated bound for
$\mb S_a(\tau)\tilde {\mb P}$.
The remaining statements are a consequence of Lemma \ref{lem:speca}.
\end{proof}

\subsection{Lipschitz bounds}
Finally, to conclude the linear theory, we prove Lipschitz bounds for the semigroup
$\mb S_a(\tau)$ and the projections from Proposition \ref{prop:linear}.
These will be needed later for the main fixed point argument.

\begin{lemma}
\label{lem:Lip}
We have the bounds
\begin{align*}
\|\mb g_a-\mb g_b\|+\|\mb h_{a,j}-\mb h_{b,j}\|&\lesssim |a-b|,\qquad j\in \{1,2,3\} \\
\|\mb P_{a}-\mb P_{b}\|+\|\mb Q_a-\mb Q_b\|&\lesssim |a-b| \\
\|\mb S_a(\tau)\tilde{\mb P}_a-\mb S_b(\tau)\tilde{\mb P}_b\|
&\lesssim |a-b|  e^{-\frac{1}{p-1}\tau} 
\end{align*}
for all $a,b\in \R^3$ small and $\tau\geq 0$, where $\mb Q_a:=\sum_{j=1}^3 \mb Q_{a,j}$.
\end{lemma}

\begin{proof}
The bounds on $\mb g_a$ and $\mb h_{a,j}$ are a consequence of the fundamental
theorem of calculus, cf.~the proof of Lemma \ref{lem:LipL'}.
 
For the bounds on the projections we use the identity
\[ \mb A^{-1}-\mb B^{-1}=\mb B^{-1}(\mb B-\mb A)\mb A^{-1} \]
which is valid for all invertible operators $\mb A,\mb B$.
Consequently, for $\lambda \in \rho(\mb L_a)\cap \rho(\mb L_b)$ we infer
\begin{align*}
 \|\mb R_{\mb L_a}(\lambda)-\mb R_{\mb L_b}(\lambda)\|&\leq \|\mb R_{\mb L_a}(\lambda)\|
\|\mb R_{\mb L_b}(\lambda)\|\|\mb L_a -\mb L_b\| \\
&\lesssim \|\mb R_{\mb L_a}(\lambda)\|
\|\mb R_{\mb L_b}(\lambda)\||a-b|.
\end{align*}
This implies the stated Lipschitz-continuity for the projections $\mb P_a$ and $\mb Q_a$.

Finally, we prove the bound for the semigroup.
For $\mb u\in \mc D(\mb L)$ we have
\[ \partial_\tau \mb S_a(\tau)\tilde{\mb P}_a \mb u
=\mb L_a \mb S_a(\tau)\tilde{\mb P}_a \mb u
=\mb L_a\tilde{\mb P}_a \mb S_a(\tau)\tilde{\mb P}_a \mb u\]
and thus,
\begin{align*}
\partial_\tau &[\mb S_a(\tau)\tilde{\mb P}_a \mb u-\mb S_b(\tau)\tilde{\mb P}_b \mb u] \\
&=\mb L_a\tilde{\mb P}_a \mb S_a(\tau)\tilde{\mb P}_a \mb u
-\mb L_b\tilde{\mb P}_b \mb S_b(\tau)\tilde{\mb P}_b \mb u \\
&=\mb L_a \tilde{\mb P}_a [\mb S_a(\tau)\tilde{\mb P}_a \mb u-\mb S_b(\tau)\tilde{\mb P}_b \mb u]
+(\mb L_a \tilde{\mb P}_a-\mb L_b \tilde{\mb P}_b)\mb S_b(\tau)\tilde{\mb P}_b \mb u.
\end{align*}
Consequently, the function
\[ \Phi_{a,b}(\tau):=\frac{\mb S_a(\tau)\tilde{\mb P}_a \mb u-\mb S_b(\tau)\tilde{\mb P}_b \mb u}{|a-b|} \]
satisfies the inhomogeneous equation
\begin{equation}
\label{eq:Phiab} \partial_\tau \Phi_{a,b}(\tau)=\mb L_a \tilde{\mb P}_a \Phi_{a,b}(\tau)
+\frac{\mb L_a \tilde{\mb P}_a-\mb L_b \tilde{\mb P}_b}{|a-b|}
\mb S_b(\tau)\tilde{\mb P}_b \mb u 
\end{equation}
with data $\Phi_{a,b}(0)=\frac{\tilde{\mb P}_a-\tilde{\mb P}_b}{|a-b|}\mb u$.
The point now is that the apparently unbounded operator 
$\mb L_a \tilde{\mb P}_a-\mb L_b \tilde{\mb P}_b$ is in fact bounded.
Indeed, we have
\[ \mb L_a \tilde{\mb P}_a=\mb L_a (\mb I- \mb P_a-\mb Q_a)=\mb L_a-\mb P_a \]
since $\mb L_a \mb P_a=\mb P_a$ and $\mb L_a \mb Q_a=\mb 0$ by Proposition \ref{prop:linear}.
Thus, we infer
\[ \mb L_a \tilde{\mb P}_a-\mb L_b \tilde{\mb P}_b=\mb L_a'-\mb L_b'+\mb P_b-\mb P_a \]
and the Lipschitz-continuity of $\mb L'_a$ and $\mb P_a$ yields the bound 
\[ \left \|\frac{\mb L_a \tilde{\mb P}_a-\mb L_b \tilde{\mb P}_b}{|a-b|} \right\|\lesssim 1 \]
for all $a,b \in \R^3$ small.
Consequently, Duhamel's principle applied to Eq.~\eqref{eq:Phiab} implies
\begin{align*}
\Phi_{a,b}(\tau)&=\mb S_a(\tau) \tilde{\mb P}_a 
\frac{\tilde{\mb P}_a-\tilde{\mb P}_b}{|a-b|}\mb u \\
&\quad +\int_0^\tau \mb S_a(\tau-\sigma)\tilde{\mb P}_a
\frac{\mb L_a \tilde{\mb P}_a-\mb L_b \tilde{\mb P}_b}{|a-b|}
\mb S_b(\sigma)\tilde{\mb P}_b \mb u\, d\sigma
\end{align*}
and we obtain
\[ \|\Phi_{a,b}(\tau)\|\lesssim (1+\tau) e^{-\frac{4/3}{p-1}\tau}\|\mb u\|
\lesssim e^{-\frac{1}{p-1}\tau} \|\mb u\|. \]
By density, this estimate holds in fact for all $\mb u\in \mc H$ and the claimed
bound follows.
\end{proof}

\section{The nonlinear theory}

\subsection{Estimates for the nonlinearity}
We show that the nonlinearity $\mb N_a$ is Lipschitz-continuous with respect to both the
argument and the parameter $a$.
To this end, we first prove a simple auxiliary estimate.

\begin{lemma}
\label{lem:auxLip}
Set $H^2_R:=\{u\in H^2(\B^3): \|u\|_{H^2(\B^3)}\leq R\}$
and fix $f\in C^2(\R)$.
Then we have the bounds
\begin{align*} 
\|f\circ u-f\circ v\|_{L^\infty(\B^3)}&\leq C_R \|u-v\|_{L^\infty(\B^3)} \\
\|f\circ u-f\circ v\|_{H^1(\B^3)}&\leq C_R \|u-v\|_{H^2(\B^3)} 
\end{align*}
for all $u,v\in H^2_R$.
\end{lemma}

\begin{proof}
We have
\begin{align*}
 f(y)-f(x)&=\int_x^y f'(t)dt=(y-x)\int_0^1 f'(x+t(y-x))dt \\
 &=(y-x)F_0(x,y)
 \end{align*}
 where $F_0\in C^1(\R^2)$.
 Replacing $f$ by $f'$ we also infer
 \[ f'(y)-f'(x)=(y-x)F_1(x,y) \]
 for a function $F_1 \in C(\R^2)$.
 By the Sobolev embedding $H^2(\B^3)\hookrightarrow  C(\overline{\B^3})$
 we obtain $\|u\|_{L^\infty(\B^3)}\leq C_R$ for all $u\in H^2_R$.
This implies $-C_R\leq u(\xi)\leq C_R$ for all $\xi \in \overline{\B^3}$
 and we obtain
 \begin{align*} 
 \|f\circ u-f\circ v\|_{L^\infty(\B^3)}&\leq \sup_{\xi\in \B^3}
 |F_0(u(\xi),v(\xi))|\|u-v\|_{L^\infty(\B^3)} \\
 &\leq \|F_0\|_{L^\infty(Q_R)}\|u-v\|_{L^\infty(\B^3)} \\
 &\lesssim C_R \|u-v\|_{L^\infty(\B^3)}
 \end{align*}
 for all $u,v \in H^2_R$, where $Q_R:=[-C_R,C_R]\times [-C_R,C_R]$.
From this estimate we immediately infer
\begin{align*}
 \|f\circ u-f\circ v\|_{L^2(\B^3)}&\lesssim \|f\circ u-f\circ v\|_{L^\infty(\B^3)}
 \lesssim C_R \|u-v\|_{L^\infty(\B^3)} \\
 &\lesssim C_R \|u-v\|_{H^2(\B^3)}.
 \end{align*}
 Similarly, for the derivative we obtain
 \begin{align*}
 \|\nabla (f\circ u-f\circ v)\|_{L^2(\B^3)} 
 &=\|(f'\circ u)\nabla u-(f'\circ v)\nabla v\|_{L^2(\B^3)} \\
 &\leq \|f'\circ u-f'\circ v\|_{L^\infty(\B^3)}\|\nabla u\|_{L^2(\B^3)} \\
 &\quad +\|f'\circ v\|_{L^\infty(\B^3)}\|\nabla u-\nabla v\|_{L^2(\B^3)} \\
 &\leq \|F_1\|_{L^\infty(Q_R)}
 \|u-v\|_{L^\infty(\B^3)}\|u\|_{H^1(\B^3)} \\
 &\quad +\|f'\|_{L^\infty(-C_R,C_R)}\|u-v\|_{H^1(\B^3)} \\
 &\lesssim C_R \|u-v\|_{H^2(\B^3)}.
 \end{align*}
\end{proof}

\begin{lemma}
\label{lem:N}
Let $\delta>0$ be sufficiently small.
Then we have the estimate
\[ \|\mb N_a(\mb u)-\mb N_b(\mb v)\|\lesssim 
(\|\mb u\|+\|\mb v\|)\|\mb u-\mb v\|+(\|\mb u\|^2+\|\mb v\|^2)|a-b| \]
for all $\mb u,\mb v\in \mc H$ with $\|\mb u\|+\|\mb v\|\leq \delta$ 
and all $a,b \in \overline{\B^3_\delta}$.
\end{lemma}

\begin{proof}
Recall that $\mb N_a(\mb u)=\mb N(\Psi_a+\mb u)-\mb N(\Psi_a)-\mb L_a' \mb u$ where
\[ \mb N(\mb v)=\left (\begin{array}{c}0 \\ v_1 |v_1|^{p-1}\end{array} \right ),\qquad
\mb L_a' \mb u=\left (\begin{array}{c}0 \\ p\psi_{a,1}^{p-1}u_1 \end{array} \right ).
\]
Thus,
\[ \mb N_a(\mb u)(\xi)=\left ( \begin{array}{c} 0 \\ 
N(\psi_{a,1}(\xi),u_1(\xi)) \end{array} \right ) \]
where 
\begin{align*} N(x_0,x):&=(x_0+x)|x_0+x|^{p-1}-x_0|x_0|^{p-1}-p|x_0|^{p-1}x \\
&=F(x_0+x)-F(x_0)-F'(x_0)x,\qquad F(y):=y|y|^{p-1}.
\end{align*}
The function $F$ is smooth on $(0,\infty)$. Consequently,
$N$ is smooth on $\Omega_q:=(q,\infty)\times (-q,\infty)$ for any $q>0$. 
Thus, for $(x_0,x)\in \Omega_q$, we may write
\begin{align*}
N(x_0,x)&=\int_{x_0}^{x_0+x}F'(t)dt-F'(x_0)x=x\int_0^1 [F'(x_0+tx)-F'(x_0)]dt \\
&=x\int_0^1 \int_{x_0}^{x_0+tx}F''(s)dsdt =x^2\int_0^1 \int_0^1 t F''(x_0+stx)dsdt \\
&=:x^2\tilde N(x_0,x)
\end{align*}
with $\tilde N$ smooth on $\Omega_q$.
Since $\delta>0$ is assumed to be small, we have 
\[ \psi_{a,1}(\xi)=c_p[A_0(a)-A_j(a)\xi^j]^{-\frac{2}{p-1}}\geq q \] 
for all
$|a|\leq \delta$, $\xi \in \overline{\B^3}$, and a suitable $q>0$
(recall that $A_0(a)=1+O(|a|)$ and $A_j(a)=O(|a|)$ for $|a|\lesssim 1$).
Consequently, the Sobolev embedding $H^2(\B^3)\hookrightarrow C(\overline{\B^3})$
implies $(\psi_{a,1}(\xi),u_1(\xi))\in \Omega_q$ for all $\xi\in \overline{\B^3}$ and
all $\mb u$ with $\|\mb u\|\leq \delta$.
Thus, we obtain
\[ \mb N_a(\mb u)(\xi)=\left ( \begin{array}{c} 0 \\
u_1(\xi)^2 \tilde N(\psi_{a,1}(\xi),u_1(\xi)) \end{array} \right ). \]
By using the 
pointwise identity
\begin{align*} u_1^2 &\tilde N(\psi_{a,1},u_1)-v_1^2 \tilde N(\psi_{a,1},v_1) \\
&=\tfrac12 (u_1+v_1)(u_1-v_1)[\tilde N(\psi_{a,1},u_1)+\tilde N(\psi_{a,1},v_1)]\\
&\quad +\tfrac12 (u_1^2+v_1^2)[\tilde N(\psi_{a,1},u_1)-\tilde N(\psi_{a,1},v_1)] 
\end{align*}
and Lemma \ref{lem:auxLip}, 
we obtain
\begin{align*}
\|u_1^2 &\tilde N(\psi_{a,1},u_1)-v_1^2 \tilde N(\psi_{a,1},v_1)\|_{L^2(\B^3)} \\
&\lesssim (\|u_1\|_{L^\infty(\B^3)}+\|v_1\|_{L^\infty(\B^3)})\|u_1-v_1\|_{L^2(\B^3)} \\
&\quad +(\|u_1\|_{L^\infty(\B^3)}^2+\|v_1\|_{L^\infty(\B^3)}^2)\|
\tilde N(\psi_{a,1},u_1)-\tilde N(\psi_{a,1},v_1)\|_{L^2(\B^3)} \\
&\lesssim (\|u_1\|_{H^2(\B^3)}+\|v_1\|_{H^2(\B^3)})\|u_1-v_1\|_{H^2(\B^3)} \\
&\lesssim (\|\mb u\|+\|\mb v\|)\|\mb u-\mb v\|.
\end{align*}
Furthermore, in order to estimate 
\[\|\nabla[u_1^2\tilde N(\psi_{a,1},u_1)
-v_1^2 \tilde N(\psi_{a,1},v_1)]\|_{L^2(\B^3)},\]
we use
\begin{align*}
\partial_{\xi^j} \tilde N(\psi_{a,1}(\xi),u_1(\xi))
&=\partial_1 \tilde N(\psi_{a,1}(\xi),u_1(\xi))\partial_j \psi_{a,1}(\xi) \\
&\quad +\partial_2 \tilde N(\psi_{a,1}(\xi),u_1(\xi))\partial_j u_1(\xi)
\end{align*}
and
\begin{align*}
 \|&\partial_2 \tilde N(\psi_{a,1},u_1)\nabla u_1
 -\partial_2 \tilde N(\psi_{a,1},v_1)\nabla v_1\|_{L^2(\B^3)} \\
 &\lesssim \|\partial_2\tilde N(\psi_{a,1},u_1)\|_{L^\infty(\B^3)}\|\nabla u_1
 -\nabla v_1\|_{L^2(\B^3)} \\
 &\quad +\|\partial_2 \tilde N(\psi_{a,1},u_1)-\partial_2 
 \tilde N(\psi_{a,1},v_1)\|_{L^\infty(\B^3)}
 \|\nabla v_1\|_{L^2(\B^3)} \\
 &\lesssim \|u_1-v_1\|_{H^2(\B^3)}
 \end{align*}
 which yields
 \begin{align*}
 \|\nabla[\tilde N(\psi_{a,1},u_1)-\tilde N(\psi_{a,1},v_1)]\|_{L^2(\B^3)}\lesssim 
 \|u_1-v_1\|_{H^2(\B^3)}.
 \end{align*}
 Now one uses the simple product estimate 
 \[ \|\nabla(fg)\|_{L^2(\B^3)}\lesssim
 \|\nabla f\|_{L^2(\B^3)}\|g\|_{L^\infty(\B^3)}+\|f\|_{L^\infty(\B^3)}\|
 \nabla g\|_{L^2(\B^3)} \]
 combined with the 
 Sobolev embedding $H^2(\B^3)\hookrightarrow L^\infty(\B^3)$ to conclude that
 \begin{align*}
  \|\nabla &[u_1^2 \tilde N(\psi_{a,1},u_1)-v_1^2 \tilde N(\psi_{a,1},v_1)]\|_{L^2(\B^3)} \\
 &\lesssim \big (\|u_1\|_{H^2(\B^3)}+\|v_1\|_{H^2(\B^3)}\big )\|u_1-v_1\|_{H^2(\B^3)} 
 \end{align*}
which yields
 \begin{align*}
\|\mb N_a(\mb u)-\mb N_a(\mb v)\|&=\|u_1^2 \tilde N(\psi_{a,1},u_1)
-v_1^2 \tilde N(\psi_{a,1},v_1)\|_{H^1(\B^3)} \\
 &\lesssim (\|\mb u\|+\|\mb v\|)\|\mb u-\mb v\|. 
 \end{align*}
 In order to complete the proof, it suffices to show that
 \[ \|\mb N_a(\mb u)-\mb N_b(\mb u)\|\lesssim \|\mb u\|^2|a-b|. \]
 To this end, we use
 \begin{align*}
 \tilde N(&\psi_{b,1}(\xi),u_1(\xi))-\tilde N(\psi_{a,1}(\xi),u_1(\xi)) \\
 &=\int_0^1 \partial_s \tilde N(\psi_{a+s(b-a),1}(\xi),u_1(\xi))ds \\
 &=(b^j-a^j)\int_0^1 \partial_1 \tilde N(\psi_{a+s(b-a),1}(\xi),u_1(\xi))\varphi_{a+s(b-a),j}(\xi)
 ds
 \end{align*}
 where $\varphi_{a,j}(\xi):=\partial_{a^j}\psi_{a,1}(\xi)$.
 From this representation and the Sobolev embedding $H^2(\B^3)
 \hookrightarrow L^\infty(\B^3)$ it follows
 that
 \begin{align*}
 \|\mb N_a(\mb u)-\mb N_b(\mb u)\|&=\|u_1^2[\tilde N(\psi_{b,1},u_1)
 -\tilde N(\psi_{a,1},u_1)]\|_{H^1(\B^3)} \\
 &\lesssim \|u_1\|_{H^2(\B^3)}^2|a-b| \\
 &\leq \|\mb u\|^2|a-b|
 \end{align*}
 which finishes the proof.
\end{proof}

\subsection{More estimates}
We apply Duhamel's principle to rewrite Eq.~\eqref{eq:Phi} in weak form as an integral equation.
For initial data $\Phi(0)=\mb u \in \mc H$ this yields
\begin{align} 
\label{eq:Phiweak}
\Phi(\tau)=&\mb S_{a_\infty}(\tau)\mb u \nonumber \\
&+\int_0^\tau \mb S_{a_\infty}(\tau-\sigma) 
\Big [\hat{\mb L}_{a(\sigma)}\Phi(\sigma)+\mb N_{a(\sigma)}(\Phi(\sigma)) 
-\partial_\sigma \Psi_{a(\sigma)} \Big ]d\sigma 
\end{align}
where we use the abbreviation $\hat{\mb L}_{a(\sigma)}:=\mb L'_{a(\sigma)}-\mb L'_{a_\infty}$.
As a preparation we derive some basic estimates for the terms involved
in Eq.~\eqref{eq:Phiweak}.
First, we introduce the following Banach spaces.
\begin{definition}
We set $\mc X:=\{\Phi \in C([0,\infty),\mc H): \|\Phi\|_{\mc X}<\infty\}$
where 
\[ \|\Phi\|_{\mc X}:=\sup_{\tau>0}[e^{\omega_p\tau}\|\Phi(\tau)\|],\qquad 
\omega_p:=\tfrac{1}{p-1}, \]
and $X:=\{a \in C^1([0,\infty),\R^3): a(0)=0, \|a\|_X<\infty\}$
where
\[ \|a\|_X:=\sup_{\tau>0}[e^{\omega_p\tau}|\dot a(\tau)|+|a(\tau)|]. \]
For $\delta>0$ we also define the closed subsets $\mc X_\delta \subset \mc X$ and
$X_\delta \subset X$ by $\mc X_\delta:=\{\Phi\in \mc X: \|\Phi\|_{\mc X}\leq \delta\}$
and $X_\delta:=\{a\in X: |\dot a(\tau)|\leq \delta e^{-\omega_p\tau}\}$.
\end{definition}

\begin{lemma}
\label{lem:est1}
Let $\delta>0$ be sufficiently small and suppose
$\Phi \in \mc X_\delta$ and $a\in X_\delta$.
Then we have the bounds
\begin{align*}
\|\hat{\mb L}_{a(\tau)}\Phi(\tau)\|+\|\mb N_{a(\tau)}(\Phi(\tau))\|
&\lesssim \delta^2 e^{-2\omega_p\tau} \\
\|\mb P_{a_\infty}\partial_\tau \Psi_{a(\tau)}\|+\|(\mb I-\mb Q_{a_\infty})
\partial_\tau \Psi_{a(\tau)}\|&\lesssim \delta^2 e^{-2\omega_p\tau}
\end{align*} 
for all $\tau\geq 0$.
\end{lemma}

\begin{proof}
First note that
\[ |a(\tau_2)-a(\tau_1)|\leq \int_{\tau_1}^{\tau_2} |\dot a(\sigma)|d\sigma 
\lesssim \delta (e^{-\omega_p\tau_1}+e^{-\omega_p \tau_2})\to 0 \]
as $\tau_1,\tau_2 \to \infty$. Consequently, $a_\infty=\lim_{\tau\to\infty} a(\tau)$ 
exists and we have 
\[ |a_\infty-a(\tau)|\leq \int_\tau^\infty |\dot a(\sigma)|d\sigma\lesssim \delta e^{-\omega_p \tau}. \]
Thus, by Lemma \ref{lem:Lip} we obtain
\[
\|\hat{\mb L}_{a(\tau)}\Phi(\tau)\|\leq
\|\mb L_{a(\tau)}'-\mb L_{a_\infty}'\|\|\Phi(\tau)\|
\lesssim \delta e^{-\omega_p\tau}|a(\tau)-a_\infty|\lesssim \delta^2 e^{-2\omega_p\tau}.
\]
Furthermore, we have
\[ \|\mb N_{a(\tau)}(\Phi(\tau))\|\lesssim \|\Phi(\tau)\|^2\lesssim \delta^2 e^{-2\omega_p\tau} \]
by Lemma \ref{lem:N}.
Finally, we note that
\[ \partial_\tau \Psi_{a(\tau)}=\dot a^k(\tau)\mb h_{a(\tau),k}=\dot a^k(\tau)
\mb h_{a_\infty,k}+\dot a^k(\tau)\hat{\mb h}_{a(\tau),k}, \]
where we set 
$\hat {\mb h}_{a(\tau),k}:=\mb h_{a(\tau),k}-\mb h_{a_\infty,k}$.
Since $\mb P_{a_\infty}\mb h_{a_\infty,k}=\mb 0$ by Proposition \ref{prop:linear}, we find
\[ \|\mb P_{a_\infty}\partial_\tau \Psi_{a(\tau)}\|\lesssim |\dot a^k(\tau)|
\|\hat{\mb h}_{a(\tau),k}\|\lesssim \delta e^{-\omega_p\tau}|a(\tau)-a_\infty|
\lesssim \delta^2 e^{-2\omega_p\tau} \]
by Lemma \ref{lem:Lip}.
Since $(\mb I-\mb Q_{a_\infty})\mb h_{a_\infty,k}=\mb 0$, the last estimate follows as well.
\end{proof}

We also prove corresponding Lipschitz-bounds.
\begin{lemma}
\label{lem:est2}
Let $\delta>0$ be sufficiently small.
Then we have the bounds
\begin{align*}
\|\hat{\mb L}_{a(\tau)}\Phi(\tau)-\hat{\mb L}_{b(\tau)}\Psi(\tau)\|
&\lesssim \delta e^{-2\omega_p\tau}\big (\|\Phi-\Psi\|_{\mc X}+\|a-b\|_X \big )\\
\|\mb N_{a(\tau)}(\Phi(\tau))-\mb N_{b(\tau)}(\Psi(\tau))\|&\lesssim
\delta e^{-2\omega_p\tau}\big (\|\Phi-\Psi\|_{\mc X}+\|a-b\|_X \big ) \\
\|\mb P_{a_\infty}\partial_\tau \Psi_{a(\tau)}-\mb P_{b_\infty}\partial_\tau \Psi_{b(\tau)}\|&\lesssim
\delta e^{-2\omega_p\tau}\|a-b\|_X 
\end{align*}
as well as
\[
\|(\mb I-\mb Q_{a_\infty})\partial_\tau \Psi_{a(\tau)}
-(\mb I-\mb Q_{b_\infty})\partial_\tau \Psi_{b(\tau)}\|\lesssim
\delta e^{-2\omega_p\tau}\|a-b\|_X  \]
for all $\Phi,\Psi \in \mc X_\delta$, $a,b \in X_\delta$, and $\tau\geq 0$.
\end{lemma}

\begin{proof}
Recall that $\hat{\mb L}_{a(\tau)}=\mb L_{a(\tau)}'-\mb L_{a_\infty}'$ where
\[ \mb L_{a(\tau)}' \mb u=\left (\begin{array}{c}0 \\ p\psi^{p-1}_{a(\tau),1}u_1 \end{array}
\right ). \]
Since $(a,\xi)\mapsto \psi_{a,1}(\xi)$ is smooth for small $a \in \R^3$,
we have
\begin{align*} p\psi_{a_\infty,1}^{p-1}(\xi)-p\psi_{a(\tau),1}^{p-1}(\xi)&=p\int_\tau^\infty
\partial_\sigma \psi_{a(\sigma),1}^{p-1}(\xi)d\sigma \\
&=\int_\tau^\infty \dot a^k(\sigma)\varphi_{a(\sigma),k}(\xi)d\sigma
\end{align*}
where $\varphi_{a,k}(\xi):=p\partial_{a^k}\psi_{a,1}^{p-1}(\xi)$.
Consequently, for $\mb u\in C^\infty(\overline{\B^3})^2$, we infer
\begin{align*}
\hat{\mb L}_{a(\tau)}&\mb u(\xi)-\hat{\mb L}_{b(\tau)}\mb u(\xi) \\
&= \left (\begin{array}{c}0 \\ u_1(\xi) 
\end{array} \right ) 
\int_\tau^\infty [\dot a^k(\sigma)\varphi_{a(\sigma),k}(\xi)
-\dot b^k(\sigma)\varphi_{b(\sigma),k}(\xi)]d\sigma 
\end{align*}
and this yields
\begin{align*}
\|&\hat{\mb L}_{a(\tau)}\mb u-\hat{\mb L}_{b(\tau)}\mb u\| \\
&\lesssim \|u_1\|_{H^1(\B^3)}\int_\tau^\infty \|\dot a^k(\sigma)\varphi_{a(\sigma),k}-\dot b^k(\sigma)
\varphi_{b(\sigma),k}\|_{W^{1,\infty}(\B^3)}d\sigma \\
&\lesssim \|\mb u\|\int_\tau^\infty |\dot a(\sigma)-\dot b(\sigma)|d\sigma
+\|\mb u\|\int_\tau^\infty |\dot a(\sigma)||a(\sigma)-b(\sigma)|d\sigma \\
&\lesssim \|\mb u\| \int_\tau^\infty e^{-\omega_p\sigma}\|a-b\|_Xd\sigma
\end{align*}
which shows that 
$\|\hat{\mb L}_{a(\tau)}-\hat{\mb L}_{b(\tau)}\|\lesssim e^{-\omega_p\tau}\|a-b\|_X$.
Thus, we find
\begin{align*} 
\|&\hat{\mb L}_{a(\tau)}\Phi(\tau)-\hat{\mb L}_{b(\tau)}\Psi(\tau)\| \\
&\lesssim \|\hat{\mb L}_{a(\tau)}-\hat{\mb L}_{b(\tau)}\|\|\Phi(\tau)\|
+\|\hat{\mb L}_{b(\tau)}\|\|\Phi(\tau)-\Psi(\tau)\| \\
&\lesssim \delta e^{-2\omega_p\tau}\big (\|a-b\|_X+\|\Phi-\Psi\|_{\mc X}\big ),
\end{align*}
as claimed.
The bound on the nonlinearity follows directly from Lemma \ref{lem:N}.
Finally, by using $\mb P_{a_\infty}\partial_\tau \Psi_{a(\tau)}=
\dot a^k(\tau)\mb P_{a_\infty}\hat{\mb h}_{a(\tau),k}$ with
$\hat{\mb h}_{a(\tau),k}=\mb h_{a(\tau),k}-\mb h_{a_\infty,k}$,
we infer
\begin{align*}
\|&\mb P_{a_\infty}\partial_\tau \Psi_{a(\tau)}-\mb P_{b_\infty}\partial_\tau \Psi_{b(\tau)}\| \\
&\lesssim \|\dot a^k(\tau)\mb P_{a_\infty}-\dot b^k(\tau)\mb P_{b_\infty}\|\|\hat{\mb h}_{a(\tau),k}\|
+\|\dot b^k(\tau)\mb P_{b_\infty}\|\|\hat{\mb h}_{a(\tau),k}-\hat{\mb h}_{b(\tau),k}\| \\
&\lesssim \delta e^{-2\omega_p\tau}\|a-b\|_X
\end{align*}
where we have used the bound
\begin{equation}
\label{eq:hhat}
 \|\hat{\mb h}_{a(\tau),k}-\hat{\mb h}_{b(\tau),k}\|\lesssim e^{-\omega_p\tau}\|a-b\|_X 
 \end{equation}
which is a consequence of the representation
\begin{align*}
\hat{\mb h}_{a(\tau),k}(\xi)=-\int_\tau^\infty \partial_\sigma \mb h_{a(\sigma),k}(\xi)d\sigma
\end{align*}
and the smoothness of $(a,\xi)\mapsto \partial_a \mb h_{a,k}(\xi)$ (for $a \in \R^3$ small).
Since $(\mb I-\mb Q_{a_\infty})\partial_\tau \Psi_{a(\tau)}
=\dot a^k(\tau)(\mb I-\mb Q_{a_\infty})\hat{\mb h}_{a(\tau),k}$, the last claim
follows as well.
\end{proof}

\subsection{The modulation equation}
Our goal is to construct a global (in $\tau$) solution of Eq.~\eqref{eq:Phiweak}.
The difficulty here is of course that the linearized evolution $\mb S_{a_\infty}$ 
has the unstable
subspaces $\rg \mb P_{a_\infty}$ and $\rg \mb Q_{a_\infty}$ which are ``induced'' by
the time-translation and Lorentz symmetry of the problem.
We will ``kill'' the Lorentz instability by modulation, i.e., 
we choose $a(\tau)$ in such a way that
this instability is suppressed. 
In order to derive an equation for $a(\tau)$, we formally apply $\mb Q_{a_\infty,j}$ to
Eq.~\eqref{eq:Phiweak}.
By Proposition \ref{prop:linear} this yields
\[ \mb Q_{a_\infty,j}\Phi(\tau)=\mb Q_{a_\infty,j}\mb u+\mb Q_{a_\infty,j}
\int_0^\tau \Big [\hat{\mb L}_{a(\sigma)}\Phi(\sigma)+\mb N_{a(\sigma)}(\Phi(\sigma)) 
-\partial_\sigma \Psi_{a(\sigma)} \Big ]d\sigma.  \]
We would like to set the right-hand side equal to zero.
This, however, is not possible since at $\tau=0$ this would entail the condition
$\mb Q_{a_\infty,j}\mb u=\mb 0$ on the initial data which is not satisfied in general.
To go around this small technicality, we choose a smooth cut-off $\chi: [0,\infty)\to [0,1]$ 
satisfying
$\chi(\tau)=1$ for $\tau \in [0,1]$, $\chi(\tau)=0$ for $\tau\geq 4$, 
and $|\chi'(\tau)|\leq 1$ for all $\tau\geq 0$.
Then we make the ansatz $\mb Q_{a_\infty,j}\Phi(\tau)=\chi(\tau)\mb h$
for a $\mb h\in \rg\mb Q_{a_\infty}$
and evaluation at $\tau=0$ yields $\mb h=\mb Q_{a_\infty,j}\mb u$.
Thus, we obtain the three modulation equations
\begin{equation}
\label{eq:mod1}
[1-\chi(\tau)]\mb Q_{a_\infty,j}\mb u+\mb Q_{a_\infty,j}
\int_0^\tau \Big [\hat{\mb L}_{a(\sigma)}\Phi(\sigma)+\mb N_{a(\sigma)}(\Phi(\sigma)) 
-\partial_\sigma \Psi_{a(\sigma)} \Big ]d\sigma=\mb 0
\end{equation}
for the three functions $a_j: [0,\infty)\to \R$, $j\in \{1,2,3\}$.
Recall that by Proposition \ref{prop:linear} we have $\mb Q_{a_\infty,j}\mb h_{a_\infty,k}
=\delta_{jk}\mb h_{a_\infty,j}$.
Thus, with $\hat{\mb h}_{a(\tau),k}=\mb h_{a(\tau),k}-\mb h_{a_\infty,k}$, we obtain
\begin{align*} 
\mb Q_{a_\infty,j}\int_0^\tau \partial_\sigma \Psi_{a(\sigma)}d\sigma&=
\mb Q_{a_\infty,j}\int_0^\tau \dot a^k(\sigma)[\mb h_{a_\infty,k}
+\hat{\mb h}_{a(\sigma),k}]\,d\sigma  \\
&=a_j(\tau)\mb h_{a_\infty,j}+\mb Q_{a_\infty,j}\int_0^\tau \dot a^k(\sigma)
\hat{\mb h}_{a(\sigma),k}\,d\sigma
\end{align*}
if we assume $a(0)=0$.
Inserting this into Eq.~\eqref{eq:mod1} we find
\begin{align}
\label{eq:mod}
a_j(\tau)\mb h_{a_\infty,j}&=
[1-\chi(\tau)]\mb Q_{a_\infty,j}\mb u \nonumber \\
&\quad +\mb Q_{a_\infty,j}
\int_0^\tau \Big [\hat{\mb L}_{a(\sigma)}\Phi(\sigma)
+\mb N_{a(\sigma)}(\Phi(\sigma)) \Big ]d\sigma \nonumber \\
&\quad -\mb Q_{a_\infty,j}\int_0^\tau \dot a^k(\sigma)\hat{\mb h}_{a(\sigma),k}\,d\sigma
\end{align}
for $j\in \{1,2,3\}$.

\subsection{Solvability of the modulation equation}
Next, we show that $a: [0,\infty)\to\R^3$ can indeed be chosen in such a way that
Eq.~\eqref{eq:mod} holds, provided that $\Phi$ satisfies a suitable smallness condition.

\begin{lemma}
\label{lem:a}
Let $\delta>0$ be sufficiently small and let $c>0$ be sufficiently large.
Furthermore, suppose $\Phi\in \mc X_\delta$.
If $\|\mb u\|\leq \frac{\delta}{c}$ then there exists a unique function $a \in X_\delta$ 
such that Eq.~\eqref{eq:mod} holds for each $j\in \{1,2,3\}$.
Furthermore, the map $\Phi \mapsto a: \mc X_\delta \subset \mc X\to X$ is Lipschitz-continuous.
\end{lemma}

\begin{proof}
We rewrite Eq.~\eqref{eq:mod} as
\begin{align}
\label{eq:Gj}
a_j(\tau)\mb h_{a_\infty,j}&=
-\int_0^\tau \chi'(\sigma)\mb Q_{a_\infty,j}\mb u\,d\sigma \nonumber \\
&\quad +\int_0^\tau \mb Q_{a_\infty,j}\Big [\hat{\mb L}_{a(\sigma)}\Phi(\sigma)
+\mb N_{a(\sigma)}(\Phi(\sigma)) \Big ]d\sigma \nonumber \\
&\quad -\int_0^\tau \dot a^k(\sigma)\mb Q_{a_\infty,j}\hat{\mb h}_{a(\sigma),k}\,d\sigma 
\nonumber \\
&=:\int_0^\tau\mb G_j(a,\Phi,\mb u)(\sigma) d\sigma
\end{align}
and we set $G_j(a,\Phi,\mb u)(\sigma):=\|\mb h_{a_\infty,j}\|^{-2}(\mb G_j(a,\Phi,\mb u)(\sigma)|
\mb h_{a_\infty,j})$.
Then we have 
\[ a(\tau)=\int_0^\tau G(a,\Phi,\mb u)(\sigma)d\sigma=:\tilde G(a,\Phi,\mb u)(\tau) \]
where $G:=(G_1,G_2,G_3)$.
Note that, by Lemma \ref{lem:Lip},
\[ \|\hat{\mb h}_{a(\tau),k}\|=\|\mb h_{a(\tau),k}-\mb h_{a_\infty,k}\|\lesssim |a(\tau)-a_\infty|\lesssim \delta e^{-\omega_p\tau}, \]
and trivially, 
$\|\chi'(\tau)\mb Q_{a_\infty,j}\mb u\|\leq \|\chi'(\tau)\mb u\|\lesssim \frac{\delta}{c} 
e^{-2\omega_p\tau}$.
Consequently, by Lemma \ref{lem:est1} we see that $\|\mb G_j(a,\Phi,\mb u)(\tau)\|
\lesssim (\frac{\delta}{c}+\delta^2)e^{-2\omega_p\tau}$ and thus,
\[ |G(a,\Phi,\mb u)(\tau)|\leq \delta e^{-2\omega_p\tau} \]
for $a\in X_\delta$, provided $\delta>0$ is sufficiently small and $c>0$ is sufficiently
large.
This shows that $a\in X_\delta$ implies $\tilde G(a,\Phi,\mb u)\in X_\delta$.
Furthermore, Lemma \ref{lem:est2} in conjunction with the bounds
\[ \|\chi'(\tau)\mb Q_{a_\infty,j}\mb u-\chi'(\tau)\mb Q_{b_\infty,j}\mb u\|
\lesssim \delta e^{-2\omega_p\tau}|a_\infty-b_\infty|\lesssim 
\delta e^{-2\omega_p\tau}\|a-b\|_X,\]
and (use Eq.~\eqref{eq:hhat})
\[
\|\dot a^k(\tau)\mb Q_{a_\infty,j}\hat{\mb h}_{a(\tau),k}
-\dot b^k(\tau)\mb Q_{b_\infty,j}\hat{\mb h}_{b(\tau),k}\|\lesssim \delta e^{-2\omega_p\tau}
\|a-b\|_X
\]
yields
\[ \|\mb G_j(a,\Phi,\mb u)(\tau)-\mb G_j(b,\Phi,\mb u)(\tau)\|\lesssim \delta e^{-2\omega_p\tau}
\|a-b\|_X \]
which implies
\[ \|\tilde G(a,\Phi,\mb u)-\tilde G(b,\Phi,\mb u)\|_X\lesssim \delta \|a-b\|_X \]
for all $a,b \in X_\delta$.
Thus, the contraction mapping principle yields the existence and uniqueness of $a\in X_\delta$ with
$a(\tau)=\tilde G(a,\Phi,\mb u)(\tau)$.

Finally, if $b(\tau)=\tilde G(b,\Psi,\mb u)(\tau)$ for $\Psi \in \mc X_\delta$, we obtain
\begin{align*}
 |\dot a(\tau)-\dot b(\tau)|&\lesssim |G(a,\Phi,\mb u)(\tau)-G(b,\Psi,\mb u)(\tau)| \\
 &\lesssim \delta e^{-2\omega_p\tau}\|\Phi-\Psi\|_{\mc X}+\delta e^{-2\omega_p\tau}\|a-b\|_X
 \end{align*}
by Lemma \ref{lem:est2}, which yields
$\|a-b\|_X\lesssim \delta \|\Phi-\Psi\|_{\mc X}$.
\end{proof}

\subsection{The time-translation instability}
Next, we turn to the unstable subspace $\rg \mb P_{a_\infty}$.
This time we proceed differently and add a correction term which stabilizes the evolution.
In order to derive this correction, we formally 
apply $\mb P_{a_\infty}$ to Eq.~\eqref{eq:Phiweak}
which yields
\begin{align*}
 \mb P_{a_\infty}\Phi(\tau)&=e^\tau \mb P_{a_\infty}\mb u \\
&\quad +e^\tau \mb P_{a_\infty}\int_0^\tau e^{-\sigma}\Big 
 [\hat{\mb L}_{a(\sigma)}\Phi(\sigma)+\mb N_{a(\sigma)}(\Phi(\sigma)) \Big] d\sigma \\
&\quad -e^\tau \mb P_{a_\infty}\int_0^\tau e^{-\sigma}\partial_\sigma \Psi_{a(\sigma)}\;d\sigma.
\end{align*}
Motivated by this we set
\begin{align*}
 \mb C(\Phi, a, \mb u):=&\mb P_{a_\infty}\mb u \\
 &+\mb P_{a_\infty}\int_0^\infty e^{-\sigma}\left [\hat{\mb L}_{a(\sigma)}\Phi(\sigma)+\mb N_{a(\sigma)}(\Phi(\sigma))
-\partial_\sigma \Psi_{a(\sigma)} \right ] d\sigma 
\end{align*}
and consider the \emph{modified equation}
\begin{align}
\label{eq:modified}
\Phi(\tau)=&\mb S_{a_\infty}(\tau)[\mb u-\mb C(\Phi, a, \mb u)] \nonumber \\
&+\int_0^\tau \mb S_{a_\infty}(\tau-\sigma) 
\Big [\hat{\mb L}_{a(\sigma)}\Phi(\sigma)+\mb N_{a(\sigma)}(\Phi(\sigma)) 
-\partial_\sigma \Psi_{a(\sigma)} \Big ]d\sigma. 
\end{align}

\begin{proposition}[Solution of the modified equation]
\label{prop:modified}
Fix a sufficiently large $c>0$ and let $\delta>0$ be sufficiently small.
If $\|\mb u\|\leq \frac{\delta}{c}$ then there exist unique functions $a \in X_\delta$ and
$\Phi \in \mc X_\delta$ 
such that Eq.~\eqref{eq:modified} holds for all $\tau\geq 0$.
\end{proposition}

\begin{proof}
We denote the right-hand side of Eq.~\eqref{eq:modified} by $\mb K(\Phi,a,\mb u)(\tau)$.
We claim that, for sufficiently small $\delta>0$, 
$\Phi\in \mc X_\delta$ implies $\mb K(\Phi,a,\mb u)\in \mc X_\delta$, where
$a \in X_\delta$ is associated to $\Phi$ via Lemma \ref{lem:a}.
To prove this, we first consider 
\begin{align*} 
\mb P_{a_\infty}\mb K(\Phi,a,\mb u)(\tau)&=-\int_\tau^\infty e^{\tau-\sigma}
\mb P_{a_\infty}\Big  [\hat{\mb L}_{a(\sigma)}\Phi(\sigma) \\
&\qquad +\mb N_{a(\sigma)}(\Phi(\sigma)) 
-\partial_\sigma \Psi_{a(\sigma)} \Big ]d\sigma. 
\end{align*}
Lemma \ref{lem:est1} yields
\[ \|\mb P_{a_\infty}\mb K(\Phi,a,\mb u)(\tau)\|\lesssim \delta^2
\int_\tau^\infty e^{\tau-\sigma}e^{-2\omega_p\sigma}d\sigma\lesssim \delta^2 e^{-2\omega_p\tau} \]
and we obtain $\mb P_{a_\infty}\mb K(\Phi,a,\mb u)\in \mc X_{\delta/4}$
if $\delta$ is sufficiently small.
 Next, by construction of $a$, we have
 \begin{align*} 
 \mb Q_{a_\infty}\mb K(\Phi,a,\mb u)(\tau)&=\chi(\tau)\mb Q_{a_\infty}\mb u
 -\mb Q_{a_\infty}\mb C(\Phi,a,\mb u)
=\chi(\tau)\mb Q_{a_\infty}\mb u
 \end{align*}
where we have used the fact that $\mb C(\Phi,a,\mb u)\in \rg \mb P_{a_\infty}$.
 This yields
 \begin{equation}
 \label{eq:QK} \|\mb Q_{a_\infty}\mb K(\Phi,a,\mb u)(\tau)\|
 \lesssim \tfrac{\delta}{c} e^{-2\omega_p\tau}
\end{equation}
which implies $\mb Q_{a_\infty}\mb K(\Phi,a,\mb u)\in \mc X_{\delta/4}$ provided
$c$ is sufficiently large.
In summary, we obtain $\|[\mb P_{a_\infty}+\mb Q_{a_\infty}]\mb K(\Phi,a,\mb u)\|\in 
\mc X_{\delta/2}$.
Finally, we consider 
\begin{align*}
[\mb I&-\mb P_{a_\infty}-\mb Q_{a_\infty}]\mb K(\Phi,a,\mb u)(\tau) \\
&=\mb S_{a_\infty}(\tau)\tilde{\mb P}_{a_\infty}[\mb u-\mb C(\Phi,a,\mb u)] \\
&\quad +\int_0^\tau \mb S_{a_\infty}(\tau-\sigma)\tilde{\mb P}_{a_\infty}
\Big [\hat{\mb L}_{a(\sigma)}\Phi(\sigma)+\mb N_{a(\sigma)}(\Phi(\sigma)) 
-\partial_\sigma \Psi_{a(\sigma)} \Big ]d\sigma.
\end{align*}
From Lemma \ref{lem:est1} we infer
\[ \|\mb C(\Phi,a,\mb u)\|\lesssim 
\tfrac{\delta}{c}+\delta^2 \int_0^\tau e^{-\sigma}e^{-2\omega_p\sigma}
d\sigma\lesssim \tfrac{\delta}{c}+\delta^2 \]
and Proposition \ref{prop:linear} and Lemma \ref{lem:est1} yield
\begin{align*} 
\|[\mb I&-\mb P_{a_\infty}-\mb Q_{a_\infty}]\mb K(\Phi,a,\mb u)(\tau)\| \\
&\lesssim (\tfrac{\delta}{c}+\delta^2) e^{-\omega_p\tau}+\delta^2 \int_0^\tau e^{-\omega_p(\tau-\sigma)}
e^{-2\omega_p\sigma}d\sigma\lesssim (\tfrac{\delta}{c}+\delta^2) e^{-\omega_p\tau}.
\end{align*}
Consequently, we see that $\Phi\in \mc X_\delta$ implies 
$\mb K(\Phi,a,\mb u)\in \mc X_\delta$,
as claimed.

Next, we claim that
\begin{equation}
\label{eq:K-K} \|\mb K(\Phi,a,\mb u)-\mb K(\Psi,b,\mb u)\|_{\mc X}
\lesssim \delta \|\Phi-\Psi\|_{\mc X} 
\end{equation}
for $\Phi,\Psi\in \mc X_\delta$ where $b \in X_\delta$ is associated to $\Psi$ via
Lemma \ref{lem:a}.
Indeed, from Lemma \ref{lem:est2} we infer
\begin{align}
\label{eq:PK-PK}
\|&\mb P_{a_\infty}\mb K(\Phi,a,\mb u)(\tau)-\mb P_{b_\infty}\mb K(\Psi,b,\mb u)(\tau)\| \nonumber \\
&\lesssim \delta \int_\tau^\infty e^{\tau-\sigma}e^{-2\omega_p\sigma}\|\Phi-\Psi\|_{\mc X} d\sigma
\lesssim \delta e^{-2\omega_p\tau}\|\Phi-\Psi\|_{\mc X}
\end{align}
where we have used $\|a-b\|_X\lesssim \|\Phi-\Psi\|_{\mc X}$ (Lemma \ref{lem:a}).
Furthermore,
\begin{align}
\label{eq:QK-QK}
\|&\mb Q_{a_\infty}\mb K(\Phi,a,\mb u)(\tau)
-\mb Q_{b_\infty}\mb K(\Psi,b,\mb u)(\tau)\| \nonumber \\
&=|\chi(\tau)|\|\mb Q_{a_\infty}\mb u-\mb Q_{b_\infty}\mb u\|
\lesssim \delta e^{-2\omega_p\tau}\|a-b\|_X \nonumber \\
&\lesssim \delta e^{-2\omega_p\tau}\|\Phi-\Psi\|_{\mc X}
\end{align}
by Lemmas \ref{lem:Lip} and \ref{lem:a}.
Finally,
\begin{align*}
\|\mb C(\Phi,a,\mb u)-\mb C(\Psi,b,\mb u)\|&\lesssim \delta \|\Phi-\Psi\|_{\mc X} \\
&\quad +\delta \int_0^\infty e^{-\sigma -2\omega_p\sigma}\|\Phi-\Psi\|_{\mc X}d\sigma \\
&\lesssim \delta \|\Phi-\Psi\|_{\mc X}
\end{align*}
and thus, by Lemmas \ref{lem:Lip} and \ref{lem:est2},
\begin{align}
\label{eq:tPK-tPK}
\|&\tilde{\mb P}_{a_\infty}\mb K(\Phi,a,\mb u)(\tau)
-\tilde{\mb P}_{b_\infty}\mb K(\Psi,b,\mb u)(\tau)\| \nonumber\\
&\lesssim \delta e^{-\omega_p\tau}\|\Phi-\Psi\|_{\mc X}
+\delta \int_0^\tau e^{-\omega_p(\tau-\sigma)}e^{-2\omega_p\sigma}
\|\Phi-\Psi\|_{\mc X}\,d\sigma \nonumber\\
&\lesssim \delta e^{-\omega_p\tau}\|\Phi-\Psi\|_{\mc X}.
\end{align}
By putting together Eqs.~\eqref{eq:PK-PK}, \eqref{eq:QK-QK}, and \eqref{eq:tPK-tPK}, we obtain
Eq.~\eqref{eq:K-K}.
Thus, the assertion is a consequence of the contraction mapping principle.
\end{proof}

\subsection{Variation of blowup time}
Our actual intention is to solve Eq.~\eqref{eq:modified} without the correction term
$\mb C(\Phi,a,\mb u)$.
So far, we can do this only in the trivial case where $\mb u=\mb 0$ (with the solution
$a(\tau)=0$, $\Phi(\tau)=\mb 0$).
The instability which is suppressed by the correction term $\mb C(\Phi,a,\mb u)$ is related
to the time-translation symmetry of the equation, i.e., to the choice of the blowup time
$T$.
Although the blowup time $T$ does not appear explicitly in the equation, it does show up in the
initial data.
Recall from
Eq.~\eqref{eq:data} and Eq.~\eqref{eq:modansatz} that the data we prescribe are 
of the form
\begin{align*} \Phi(0)(\xi)&=\Psi(0)(\xi)-\Psi_{a(0)}(\xi) \\
&=\left (\begin{array}{c}
T^{\frac{2}{p-1}}[\psi_{0,1}(T\xi)+\tilde f(T\xi)] \\
T^{\frac{p+1}{p-1}}[\psi_{0,2}(T\xi)+\tilde g(T\xi)] \end{array} \right )-\Psi_0(\xi) 
\end{align*}
for some fixed, given functions $\tilde f,\tilde g$.
For any $\mb v\in \mc H$ we define the rescaling
\[ \mb v^T(\xi):=\left (\begin{array}{c}
T^{\frac{2}{p-1}}v_1(T\xi)\\
T^{\frac{p+1}{p-1}}v_2(T\xi) \end{array} \right ) \]
and set
\[ \mb U(T,\mb v):=\mb v^T+\Psi_0^T-\Psi_0. \]
For $\mb v=(\tilde f,\tilde g)$ we may then rewrite the initial data as
\begin{equation}
\label{eq:datavec}
\Phi(0)=\mb U(T,\mb v).
\end{equation}
The advantage of this notation is that the profile $\mb v$ is independent of $T$.
Thus, the dependence of the data on the fixed profile $\mb v$ on the one hand, and the
blowup time $T$ on the other hand is now clearly separated.

We will show that $T$ can be chosen in such a way that the correction term vanishes.
As a preparation we need the following technical result.
For $U\subset \R^3$ we set
\[ \mc H(U):=H^2(U)\times H^1(U). \]

\begin{lemma}
\label{lem:U}
Let $\delta>0$ be sufficiently small.
If $\mb v\in \mc H(\B^3_{1+\delta})$ satisfies
$\|\mb v\|_{\mc H(\B^3_{1+\delta})}\leq \delta$ then we have
\[ \|\mb U(T,\mb v)\|\lesssim \delta \]
for all $T\in [1-\delta,1+\delta]$.
Furthermore, the map $T\mapsto \mb U(T,\mb v): [1-\delta,1+\delta]\to \mc H$ is 
continuous.
\end{lemma}

\begin{proof}
By definition, we have
\[ [\mb U(T,\mb v)]_1(\xi)=T^\frac{2}{p-1}v_1(T\xi)
+T^{\frac{2}{p-1}}\psi_{0,1}(T\xi)-\psi_{0,1}(\xi). \]
Since $\psi_{0,1}\in C^\infty(\R^3)$, the fundamental theorem of calculus implies 
$\|T^\frac{2}{p-1}\psi_{0,1}(T\,\cdot)-
\psi_{0,1}\|_{H^2(\B^3)}\lesssim |T-1|$. Consequently, we infer
\begin{align*} 
\|&T^\frac{2}{p-1}v_1(T\,\cdot)+T^\frac{2}{p-1}\psi_{0,1}(T\,\cdot)-\psi_{0,1}\|_{H^2(\B^3)}  \\
&\lesssim T^\frac{2}{p-1}\|v_1(T\,\cdot)\|_{H^2(\B^3)}+|T-1| \\
&\lesssim \|v_1\|_{H^2(\B^3_T)}+\delta \\
&\lesssim \delta
\end{align*}
for all $T\in [1-\delta,1+\delta]$.
The same argument can be used for the second component of $\mb U(T,\mb v)$ and we obtain
$\|\mb U(T,\mb v)\|\lesssim \delta$, as claimed.
The second statement is a consequence of the continuity of the map 
$T\mapsto \|f(T\,\cdot)\|_{H^k(\B^3)}$ for (fixed) $f\in H^k(\B^3_{1+\delta})$.
The latter follows easily by the triangle inequality and an approximation argument
using the density of $C^\infty(\overline{\B^3_{1+\delta}})$ in $H^k(\B^3_{1+\delta})$.
\end{proof}

From this result we immediately obtain the following corollary.

\begin{corollary}
\label{cor:solmodT}
Let $c>0$ be sufficiently large and choose $\delta>0$ sufficiently small. 
Suppose $\|\mb v\|_{\mc H(\B^3_{1+\delta/c})}\leq \frac{\delta}{c}$ and $T\in [1-\frac{\delta}{c},
1+\frac{\delta}{c}]$.
Then, Eq.~\eqref{eq:modified} with
$\mb u$ replaced by $\mb U(T,\mb v)$, i.e., 
\begin{align}
\label{eq:modifiedU}
\Phi(\tau)=&\mb S_{a_\infty}(\tau)[\mb U(T,\mb v)-\mb C(\Phi, a, \mb U(T, \mb v))] \nonumber \\
&+\int_0^\tau \mb S_{a_\infty}(\tau-\sigma) 
\Big [\hat{\mb L}_{a(\sigma)}\Phi(\sigma)+\mb N_{a(\sigma)}(\Phi(\sigma)) 
-\partial_\sigma \Psi_{a(\sigma)} \Big ]d\sigma,
\end{align}
has a solution $(\Phi,a)\in \mc X_\delta\times X_\delta$. Furthermore, the map
$T\mapsto (\Phi,a): [1-\frac{\delta}{c},1+\frac{\delta}{c}]\to \mc X\times X$ is
continuous.
\end{corollary}

Next, we will show that the equation $\mb C(\Phi,a,\mb U(T,\mb v))=\mb 0$ is equivalent
to a fixed point problem of the form
$T-1=F(T)$ where $F$ is continuous and satisfies $|F(T)|\leq \frac{\delta}{c}$.

\begin{lemma}
\label{lem:modifiedU}
Let $c>0$ be sufficiently large and choose $\delta>0$ sufficiently small. 
Suppose $\|\mb v\|_{\mc H(\B^3_{1+\delta/c})}\leq \frac{\delta}{c^2}$.
Then there exist $T\in [1-\frac{\delta}{c},1+\frac{\delta}{c}]$ and functions
$(\Phi,a)\in \mc X_\delta \times X_\delta$ such that Eq.~\eqref{eq:modifiedU}
is satisfied with $\mb C(\Phi,a,\mb U(T,\mb v))=\mb 0$.
\end{lemma}

\begin{proof}
We first note that 
\[ \partial_T \Psi_0^T(\xi)|_{T=1}=\kappa_p \mb g_0(\xi) \]
for a suitable constant $\kappa_p$.
Thus, we may write 
\[ \mb U(T,\mb v)=\mb v^T+(T-1)\kappa_p\mb g_0+(T-1)^2 \mb f_T \]
where $\|\mb f_T\|\lesssim 1$ for all $T\in [1-\frac{\delta}{c},1+\frac{\delta}{c}]$.
Let $(\Phi,a)\in \mc X_\delta \times X_\delta$ be the functions associated to $T$
via Corollary \ref{cor:solmodT}.
We write
\[ \mb U(T,\mb v)=\mb v^T+(T-1)\kappa_p \mb g_{a_\infty}+(T-1)\kappa_p (\mb g_0-\mb g_{a_\infty})
+(T-1)^2 \mb f_T. \]
Since $|a_\infty-0|=|a_\infty-a(0)|\lesssim \delta$, we infer
$\|\mb g_0-\mb g_{a_\infty}\|\lesssim \delta$
and this yields
\begin{align*}
(\mb P_{a_\infty}\mb U(T,\mb v)|\mb g_{a_\infty})&=O(\tfrac{\delta}{c^2})+\kappa_p \|\mb g_{a_\infty}\|^2(T-1)
+O(\delta^2) \\
&=\kappa_p \|\mb g_{a_\infty}\|^2(T-1)+O(\tfrac{\delta}{c^2})
\end{align*}
where the $O$-terms are continuous functions of $T$.
Thus, from Lemma \ref{lem:est1} we obtain
\[ (\mb C(\Phi,a,\mb U(T,\mb v))|\mb g_{a_\infty})=
\kappa_p \|\mb g_{a_\infty}\|^2(T-1)+O(\tfrac{\delta}{c^2}) \]
and this shows that $(\mb C(\Phi,a,\mb U(T,\mb v))|\mb g_{a_\infty})=0$ is equivalent
to
\[ T-1=F(T) \]
for a function $F$ which is continuous on $[1-\tfrac{\delta}{c},1+\frac{\delta}{c}]$
and satisfies $|F(T)|\leq \frac{\delta}{c}$, provided $c>0$ is chosen large enough.
Thus, $1+F$ is a continuous self-map of the closed interval 
$[1-\frac{\delta}{c},1+\frac{\delta}{c}]$ and such a map necessarily has a fixed point.
Consequently, there exists a $T \in [1-\frac{\delta}{c},1+\frac{\delta}{c}]$ 
such that $(\mb C(\Phi,a,\mb U(T,\mb v))|\mb g_{a_\infty})=0$ and, since
$\mb C(\Phi,a,\mb U(T,\mb v))\in \langle \mb g_{a_\infty}\rangle$, we infer
$\mb C(\Phi,a,\mb U(T,\mb v))=\mb 0$, as claimed.
\end{proof}

\subsection{Proof of Theorem \ref{thm:main}}
Theorem \ref{thm:main} is now a consequence of Lemma \ref{lem:modifiedU}.
All we have to do is translate back the statement of Lemma \ref{lem:modifiedU} to
the original setting. More precisely, let $\delta,c$ be such that Lemma \ref{lem:modifiedU}
holds, and set $\delta':=\delta/c$.
Suppose the data $(\tilde f,\tilde g)\in H^2\times H^1(\B^3_{1+\delta'})$ satisfy
$\|(\tilde f,\tilde g)\|_{H^2\times H^1(\B_{1+\delta'}^3)}\leq \tfrac{\delta'}{c}$.
Set $\mb v:=(\tilde f,\tilde g)$.
Then we have
\[ \|\mb v\|_{\mc H(\B^3_{1+\delta/c})}
=\|(\tilde f,\tilde g)\|_{H^2\times H^1(\B^3_{1+\delta'})}\leq 
\tfrac{\delta}{c^2} \]
and by Lemma \ref{lem:modifiedU} we obtain a $T \in [1-\delta',1+\delta']$
and functions $(\Phi,a)\in \mc X_{c\delta'}\times X_{c\delta'}$ that solve
Eq.~\eqref{eq:Phiweak} with data $\Phi(0)=\mb U(T,\mb v)$.
This means that $\Psi(\tau):=\Psi_{a(\tau)}+\Phi(\tau)$ is a solution (in the Duhamel
sense) of Eq.~\eqref{eq:Psi} with data $\Psi(0)=\Psi_0+\mb U(T, \mb v)$.
Consequently, 
\[ u(t,x)=(T-t)^{-\frac{2}{p-1}}\psi_1\left (-\log(T-t)+\log T,\tfrac{x}{T-t}\right ) \]
solves the original wave equation \eqref{eq:main} with data
\begin{align*}
 u(0,x)&=T^{-\frac{2}{p-1}}\psi_1\left (0,\tfrac{x}{T}\right )
 =\psi_{0,1}(x)+\tilde f(x)=u_{1,0}(0,x)+\tilde f(x)  \\
 \partial_0 u(0,x)&=T^{-\frac{p+1}{p-1}}\psi_2\left (0,\tfrac{x}{T}\right )
 =\psi_{0,2}(x)+\tilde g(x)=\partial_0 u_{1,0}(0,x)+\tilde g(x)
 \end{align*}
 for $x\in \B_{1+\delta'}^3$.
 We infer
 \begin{align*}
 (T&-t)^{\frac{2}{p-1}+\frac12}\|u(t,\cdot)-u_{T,a_\infty}(t,\cdot)\|_{\dot H^2(\B^3_{T-t})} \\
 &=(T-t)^\frac12\|(\psi_1-\psi_{a_\infty,1})(-\log(T-t)+\log T, \tfrac{\cdot}{T-t})\|_{\dot H^2(\B^3_{T-t})}\\
 &=\|(\psi_1-\psi_{a_\infty,1})(-\log(T-t)+\log T, 
 \cdot)\|_{\dot H^2(\B^3)} \\
 &\lesssim \|(\Psi-\Psi_{a_\infty})(-\log(T-t)+\log T)\| \\
 &\lesssim \|\Phi(-\log(T-t)+\log T)\|+\|\Psi_{a(-\log(T-t)+\log T)}-\Psi_{a_\infty}\| \\
 &\lesssim (T-t)^{\frac{1}{p-1}}
 \end{align*}
 for all $t\in [0,T)$.
 The other bounds follow by scaling.

\appendix
\section{Angular momentum decomposition}
\label{sec:apx}
\noindent For a function\footnote{We restrict our discussion here to $d\geq 3$ since
$d=1$ is trivial and $d=2$ is complex analysis.} 
$f: \S^{d-1}\to \C$ we define its canonical extension
$f^*: \R^d\backslash \{0\} \to \C$ by $f^*(x):=f(\frac{x}{|x|})$.
We say that $f\in C^k(\S^{d-1})$ if $f^*\in C^k(\R^d\backslash\{0\})$.
For $f\in C^1(\S^{d-1})$ we write 
$\slashed\partial_j f:=\partial_j f^*$ 
and call $\slashed\partial_j f$ an \emph{angular
derivative}.
The Laplace-Beltrami operator on $\S^{d-1}$ can then be written as 
$-\slashed \partial^j \slashed \partial_j$ and we have the integration by parts
formula
\[ (\slashed\partial^j\slashed\partial_j f|g)_{\S^{d-1}}=
-(\slashed\partial_j f|\slashed\partial^j g)_{\S^{d-1}} \]
for $f,g\in C^2(\S^{d-1})$, where
\begin{align*}
(f|g)_{\S^{d-1}}&=\int_{\S^{d-1}}f(\omega)\overline{g(\omega)}
d\sigma(\omega) \\
&=\int_{\R^{d-1}}f(\psi(y))\overline{g(\psi(y))}\left (\frac{2}{|y|^2+1}
\right )^{d-1}dy 
\end{align*}
with the stereographic projection $\psi: \R^{d-1}\to \S^{d-1}\backslash
\{e_d\}$,
\[ \psi(y)=\left (\frac{2y}{|y|^2+1},\frac{|y|^2-1}{|y|^2+1}\right ). \]
We denote by $Y_{\ell,m}\in L^2(\S^{d-1})$ 
the usual $L^2(\S^{d-1})$-normalized spherical
harmonics, i.e., $Y_{\ell,m}$ is an eigenfunction of $-\slashed\partial^j
\slashed\partial_j$ with eigenvalue $\ell(\ell+d-2)$, $\ell\in \N_0$,
and $(Y_{\ell,m}|Y_{\ell',m'})_{\S^{d-1}}=\delta_{\ell\ell'}\delta_{mm'}$.
For each $\ell \in \N_0$, the eigenspace associated to the eigenvalue
$\ell(\ell+d-2)$ is finite-dimensional and we denote by 
$\Omega_{d,\ell}\subset \Z$
the set of admissible indices $m$.
The Sobolev space $H^k(\S^{d-1})$, $k\in \N_0$, 
is defined as the completion of
$C^k(\S^{d-1})$ with respect to  
\[ \|f\|_{H^k(\S^{d-1})}:=\sum_{|\alpha|\leq k}
\|\slashed \partial^\alpha f\|_{L^2(\S^{d-1})} \]
with the usual multi-index notation
\[ \alpha \in \N_0^d,\quad |\alpha|:=\sum_{j=1}^d \alpha_j,
\quad\slashed \partial^\alpha:=\slashed\partial_1^{\alpha_1}
\partial_2^{\alpha_2}\dots
\slashed\partial_d^{\alpha_d}. \]
Now we recall the following fundamental expansion result.

\begin{lemma}
\label{lem:exp}
Let $f\in C^\infty(\S^{d-1})$, $d\geq 3$, and set
\[ S_n f(\omega):=\sum_{\ell=0}^n \sum_{m\in \Omega_{d,\ell}}
(f|Y_{\ell,m})_{\S^{d-1}}Y_{\ell,m}(\omega). \]
Then we have
\[ \lim_{n\to \infty}\|S_n f-f\|_{H^k(\S^{d-1})}=0 \]
for any $k\in \N_0$.
\end{lemma}

\begin{proof}
For $k=0$ the result is classical and can be found in e.g.~\cite{AtkHan12}.
Now note that
\begin{align*}
\slashed\partial^j\slashed\partial_j S_n f&=\sum_{\ell=0}^n \sum_{m\in 
\Omega_{d,\ell}}(f|Y_{\ell,m})_{\S^{d-1}}\slashed\partial^j
\slashed\partial_j Y_{\ell,m} \\
&=-\sum_{\ell=0}^n \sum_{m\in 
\Omega_{d,\ell}}\ell(\ell+d-1)(f|Y_{\ell,m})_{\S^{d-1}} Y_{\ell,m} \\
&=\sum_{\ell=0}^n \sum_{m\in 
\Omega_{d,\ell}}(f|\slashed\partial^j\slashed\partial_j Y_{\ell,m})_{\S^{d-1}} Y_{\ell,m} \\
&=\sum_{\ell=0}^n \sum_{m\in 
\Omega_{d,\ell}}(\slashed\partial^j\slashed\partial_jf|Y_{\ell,m})_{\S^{d-1}} Y_{\ell,m}
\end{align*}
which shows that $\slashed\partial^j\slashed\partial_j S_n f
=S_n \slashed\partial^j\slashed\partial_j f$.
Consequently, we find
\[ \lim_{n\to\infty}\|\slashed\partial^j\slashed\partial_j 
(S_n f-f)\|_{L^2(\S^{d-1})}
=\lim_{n\to\infty}\| 
S_n \slashed\partial^j\slashed\partial_jf-\slashed\partial^j\slashed\partial_j f\|_{L^2(\S^{d-1})}=0. \]
For $g_n:=S_nf-f$ this yields
\begin{align*} \|g_n\|_{H^1(\S^{d-1})}^2&\lesssim \|g_n\|_{L^2(\S^{d-1})}^2+(\slashed\partial^j g_n|
\slashed\partial_j g_n)_{\S^{d-1}} \\
&=\|g_n\|_{L^2(\S^{d-1})}^2-
(g_n|\slashed\partial^j\slashed\partial_j g_n)_{\S^{d-1}} \\
&\lesssim \|g_n\|_{L^2(\S^{d-1})}^2+\|\slashed\partial^j\slashed\partial_j
g_n\|_{L^2(\S^{d-1})}^2 \to 0
\end{align*}
as $n\to\infty$, which proves the claim for $k=1$.
Now one proceeds inductively.
\end{proof}

\begin{lemma}
\label{lem:expfull}
Let $f\in C^\infty(\overline{\B^d})$, $d\geq 3$, and define $P_{\ell,m}f: [0,1]\to
\C$ by
\[ P_{\ell,m}f(r):=(f(r\,\cdot)|Y_{\ell,m})_{\S^{d-1}}
=\int_{\S^{d-1}}f(r\omega)\overline{Y_{\ell,m}(\omega)}
d\sigma(\omega). \]
Furthermore, set
\[ \tilde S_n f(x) := \sum_{\ell=0}^n \sum_{m\in \Omega_{d,\ell}}
P_{\ell,m}f(|x|)Y_{\ell,m}(\tfrac{x}{|x|}) \]
for $x\in \overline{\B^d}\backslash \{0\}$.
Then we have
\[ \lim_{n\to\infty}\|\tilde S_n f-f\|_{H^k(\B^d)}=0 \]
for any $k\in \N_0$.
\end{lemma}

\begin{proof}
We define $\varphi_n: (0,1]\to \R$ by $\varphi_n(r)
:=\|\tilde S_n f(r\,\cdot)-f(r\,\cdot)\|_{L^2(\S^{d-1})}^2$. 
From Lemma \ref{lem:exp} we have $\varphi_n(r)\to 0$ as $n\to\infty$
for each $r\in (0,1]$.
Furthermore, Bessel's inequality implies
\begin{align*}
|\varphi_n(r)|&\lesssim \|\tilde S_n f(r\,\cdot)\|_{L^2(\S^{d-1})}^2
+\|f(r\,\cdot)\|_{L^2(\S^{d-1})}^2 \\
&=\sum_{\ell=0}^n \sum_{m\in \Omega_{d,\ell}}|P_{\ell,m}f(r)|^2
+\|f(r\,\cdot)\|_{L^2(\S^{d-1})}^2 \\
&\lesssim \|f(r\,\cdot)\|_{L^2(\S^{d-1})}^2
\end{align*}
and thus, the dominated convergence theorem yields
\begin{align*}
\|\tilde S_n f-f\|_{L^2(\B^d)}^2&=\int_0^1 \varphi_n(r)r^{d-1}dr
\to 0
\end{align*}
as $n\to\infty$. This is the claim for $k=0$.

In order to prove the case $k=1$, we set $g_n:=\tilde S_nf -f$ and
note that 
\[\partial_j g_n(r\omega)=
\omega_j \partial_r g_n(r\omega)+\tfrac{1}{r}\slashed\partial_{\omega^j}
g_n(r\omega). \]
For notational convenience we define $D_\mathrm{rad}f(x):=\frac{x^j}{|x|}
\partial_j f(x)$.
Then we have $D_\mathrm{rad}f(r\omega)=\omega^j \partial_j f(r\omega)
=\partial_r f(r\omega)$.
Furthermore, we write 
\[ \slashed D_j f(x):=|x|\left (\delta_j{}^k-\frac{x_jx^k}{|x|^2}\right )
\partial_k f(x) \]
so that $\slashed\partial_{\omega^j}f(r\omega)=\slashed D_j f(r\omega)$.
We also set $\slashed \nabla f:=(\slashed D_1 f,\slashed D_2 f,
\dots, \slashed D_d f)$.
With this notation we obtain
\begin{equation} 
\label{eq:H1gn}
\|g_n\|_{H^1(\B^d)}\lesssim \|g_n\|_{L^2(\B^d)}
+\|D_\rad g_n\|_{L^2(\B^d)}
+\|\tfrac{1}{|\cdot|}
\slashed \nabla g_n\|_{L^2(\B^d)}. 
\end{equation}
We have already shown above that $\|g_n\|_{L^2(\B^d)}\to 0$ as $n\to\infty$.
Now we turn to the second term.
The dominated convergence theorem implies
\begin{align*} \partial_r P_{\ell,m}f(r)&=\partial_r \int_{\S^{d-1}}f(r\omega)
\overline{Y_{\ell,m}(\omega)}d\sigma(\omega) \\
&= \int_{\S^{d-1}}\partial_r f(r\omega)
\overline{Y_{\ell,m}(\omega)}d\sigma(\omega) 
\end{align*}
and thus, $\partial_r P_{\ell,m}f(r)=P_{\ell,m}D_\mathrm{rad}f(r)$.
As a consequence, we obtain $D_\rad \tilde S_n f=\tilde S_n D_\rad f$
which yields
\begin{align*}
 \|D_\rad (\tilde S_n f-f)\|_{L^2(\B^d)}=\|\tilde S_n D_\rad f-D_\rad f\|_{L^2(\B^d)}
 \to 0\qquad (n\to\infty)
\end{align*}
by the same argument as above.
Hence, in view of \eqref{eq:H1gn} it remains to show that 
$\|\tfrac{1}{|\cdot|}
\slashed \nabla g_n\|_{L^2(\B^d)} \to 0$.
As in the proof of Lemma \ref{lem:exp} we have
$\slashed D^j \slashed D_j \tilde S_n f=\tilde S_n \slashed D^j \slashed D_j f$
and thus,
\begin{align*} 
\tfrac{1}{r^2}\|\slashed\nabla g_n(r\,\cdot)\|_{L^2(\S^{d-1})}^2
&=\tfrac{1}{r^2}(\slashed D^j g_n(r\,\cdot)|\slashed D_j g_n(r\,\cdot))_{\S^{d-1}} \\
&=-(g_n(r\,\cdot)|\tfrac{1}{r^2}
\slashed D^j\slashed D_j g_n(r\,\cdot))_{\S^{d-1}} \\
&\leq \|g_n(r\,\cdot)\|_{L^2(\S^{d-1})}^2 \\
&\quad +\tfrac{1}{r^4}\|\tilde S_n \slashed D^j
\slashed D_j f(r\,\cdot)
-\slashed D^j\slashed D_j f(r\,\cdot)\|_{L^2(\S^{d-1})}^2.
\end{align*}
We define $\psi_n: (0,1]\to \R$ by 
\[ \psi_n(r):=\tfrac{1}{r^4}\|\tilde S_n \slashed D^j
\slashed D_j f(r\,\cdot)
-\slashed D^j\slashed D_j f(r\,\cdot)\|_{L^2(\S^{d-1})}^2 \]
and from Lemma \ref{lem:exp} we know that $\psi_n(r)\to 0$ as $n\to\infty$
for any $r\in (0,1]$.
Furthermore, from Bessel's inequality we infer
\[ |\psi_n(r)|\lesssim 
\tfrac{1}{r^4}\|\slashed D^j\slashed D_jf(r\,\cdot)\|_{L^2(\S^{d-1})}^2 \]
and by the definition of $\slashed D_j$ we have
\[ |\tfrac{1}{r^2}\slashed D^j \slashed D_j f(r\omega)|
\lesssim \sum_{|\alpha|\leq 2}|\partial^\alpha f(r\omega)| \]
which implies
\[ |\psi_n(r)|\lesssim \|f\|_{W^{2,\infty}(\B^d)}<\infty \]
for all $r\in (0,1]$.
Thus, by dominated convergence we obtain
\[ \int_0^1 \psi_n(r)r^{d-1}dr \to 0 \]
as $n\to\infty$ and this proves the desired
$\|\tfrac{1}{|\cdot|}
\slashed \nabla g_n\|_{L^2(\B^d)} \to 0$.
The general case $k\geq 2$ follows inductively.
\end{proof}

\bibliography{wave7}
\bibliographystyle{plain}

\end{document}